\documentclass[11pt,reqno,twoside]{article}
\usepackage{mysty}
\usepackage{bbm}
\usepackage{wrapfig}
\usepackage{float}
\usepackage{patchcmd}

\usepackage{fullpage}
\usepackage{enumitem}

\SetLabelAlign{parright}{\parbox[t]{\labelwidth}{\raggedleft{#1}}}
\setlist[description]{style=multiline,topsep=4pt,align=parright}%,font=\normalfont

\makeatletter
\let\reftagform@=\tagform@
\def\tagform@#1{\maketag@@@{(\ignorespaces\textcolor{black}{#1}\unskip\@@italiccorr)}}
\newcommand{\iref}[1]{\textup{\reftagform@{\tcr{\ref{#1}}}}}
\makeatother

\newcommand{\alow}{\underline{\alpha}}
\newcommand{\aup}{\overline{\alpha}}
\newcommand{\tlow}{\underline{\tau}}

\newcommand{\Cupb}{\overline{C}}

\newcommand{\hup}{\overline{h}}
\newcommand{\w}{w}
\renewcommand{\norm}[1]{\left\lVert#1\right\rVert}
\newcommand{\eas}{\overset{\textrm{a.s.}}{=}}
\newcommand{\toas}{\overset{\textrm{a.s.}}{\to}}
\newcommand{\leqas}{\overset{\textrm{a.s.}}{\leq}}

\newcommand{\seq}[1]{\left(#1\right)}
\newcommand{\reals}{\mathbb{R}}
\newcommand{\Expectation}{\mathbb{E}}

\begin{document}
%%%%%%%%%%%%%%%%%%%%%%%%%%%%%%%%%%%
\title{Zeroth order optimization with orthogonal random directions}
\author{David Kozak\thanks{Solea Energy. E-mail: dkozak@soleaenergy.com} \and
Cesare Molinari\thanks{Istituto Italiano di Tecnologia. E-mail: cecio.molinari@gmail.com} \and 
Lorenzo Rosasco\thanks{ MaLGa, DIBRIS, Università di Genova
CBMM, MIT
Istituto Italiano di Tecnologia. E-mail: lorenzo.rosasco@unige.it}
\and Luis Tenorio \thanks{Department of Applied Mathematics and Statistics, Colorado School of Mines. E-mail: ltenorio@mines.edu} \and Silvia Villa \thanks{ MaLGa, DIMA, Università degli Studi di Genova. E-mail: villa@dima.unige.it}
}
\date{}
 \maketitle
% \tableofcontents
\begin{flushleft}\end{flushleft}
\begin{abstract}
We propose and analyze  a randomized zeroth-order approach based on approximating the exact gradient by finite differences computed in a set of orthogonal random directions that changes with each iteration. A number of previously proposed methods are recovered as special cases including spherical smoothing, coordinate descent, as well as  discretized gradient descent. Our main contribution is proving convergence guarantees as well as convergence rates under different parameter choices and assumptions. In particular, we consider convex objectives, but also possibly non-convex objectives satisfying the  Polyak-Łojasiewicz (PL) condition. Theoretical results are complemented and illustrated by numerical experiments. 
\end{abstract}
%\tableofcontents

\section{Introduction}
It is common in engineering, economics, statistics, and machine learning  to try to minimize  a function for which no analytical form is readily accessible, and  only a zeroth-order oracle giving the function value at a given point is available. Sometimes an analytical form for the function exists, but the gradient does not have an explicit expression or it is infeasible to obtain. Specific examples of both scenarios are highlighted in the first chapter of \cite{conn2009introduction}, to which we add reinforcement learning \cite{salimans2017evolution,NIPS2018_7451,pmlr-v80-choromanski18a,MR2298287}. 
% This algorithm (or lightly-tweaked variations on this) is popular in reinforcement learning . The focus in these papers is no theoretical discussion but . An exception is \cite{MR2298287} which provides results similar to, but slightly worse than, \cite{duchi2015optimal}.
When only function evaluations are available the options for optimization are somewhat limited. Some of the earliest theoretical analysis of zeroth order optimization was done on \emph{random search}  \cite{matyas1965random,rastrigin1963convergence}. In random search, a step is proposed in a randomly chosen direction and the objective function is computed; if it provides an improvement over the function value at the current position then the step is taken, otherwise a new random sample is drawn and the process is repeated. 
This approach has the downside of not using derivative information, therefore relying on inexpensive function evaluations and converging slowly due to the uninformed search directions. Older methods, such as discrete gradient descent and discrete coordinate descent \cite{MR1264365,MR50243},  have similar aims but actually use approximate derivative information to determine how far to step.% , thereby making faster progress to a local optimum.  
Extensions to functions that are differentiable almost-everywhere have been considered in \cite{gupal1977method,gupal1977algorithm}. There are alternative algorithms that behave similarly to (randomized) finite difference coordinate descent but use a random basis for choosing descent directions \cite{MR499560,nesterov2017random,duchi2015optimal,MR2298287} and ample experimental results showcase the practical utility of these approaches \cite{salimans2017evolution,NIPS2018_7451,pmlr-v80-choromanski18a}. The Itoh–Abe discrete gradient method studied in \cite{grimm2017discrete,riis2018geometric,ehrhardt2018geometric} is another approach to zeroth-order optimization, but at each iteration and for each coordinate it requires to solve a scalar equation involving the objective function. Similar considerations apply for its stochastic version, where the stepsize is still defined in an implicit way.
% In this paper we consider a general approach and derive  new and unifying convergence analyses of several such methods. 

Our approach is based  on approximating the exact gradient by finite differences computed in a set of orthogonal random directions. Different randomized projection can be considered, recovering different approaches as special cases. Indeed, our general approach recovers  finite difference versions of coordinate descent, descent in subspaces defined by random orthogonal matrices, and spherical smoothing as special cases. All these methods can be treated in a unified way with our approach.  Our main contribution is  proving convergence results as well as convergence rates. For convex objectives,  we study convergence  in  function value and give realizable  conditions for  convergence of the iterates. Beyond convexity, we consider the case when  the objective satisfies the Polyak-Łojasiewicz (PL) condition for which we can give stronger guarantees with faster rates.  The PL condition is necessarily satisfied by functions that are strongly convex in which case our results are trivially extended to convergence of the iterates to the unique minimizer. For many   instances of our general algorithm the derived results are new. Simple experiments are also provided to show that the rates described by the theorems are achievable in practice.
 
The  paper is organized as follows. In Section~\ref{sec:setting}, we describe the general setting,   proposed approach and its special cases. In Section~\ref{sec:main}, we summarize and discuss our main results. 
In Section~\ref{sec:pre}, we present useful preliminary  results.
In Section~\ref{Sect: Convex case} and \ref{Sect: PL}, we detail and prove the results for the convex case and for objective functions satisfying a PL condition - respectively. In Section~\ref{sec:num} we present numerical experiments and conclude in Section~\ref{sec:conc} with some remarks and open questions.

\section{Problem statement}\label{sec:setting}
We consider the problem of finding a point $x_*\in\R^d$ such that $f(x_*)= f_*$, where
\begin{equation}\label{eq:P}
\tag{P}
f_* := \inf_{x\in \R^d} f(x).
\end{equation}
The function $f\colon \  \R^d\to\R$ satisfies the following hypotheses:\\
\begin{enumerate}[label={\rm ({\bf H.\arabic{*}})}, ref={\rm\bf H.\arabic{*}}, leftmargin=4em]
\item \label{H_lip} $f_*>-\infty$, $f$ is differentiable and $\nabla f$ is $\lambda$-Lipschitz; that is, for every $x_1$ and $x_2\in\R^d$,
$$\norm{\nabla f(x_2)-\nabla f(x_1)}{}\leq \lambda \norm{x_2-x_1}{}.$$
\item \label{H_existence} $f$ is convex and has a minimizer in $\R^d$.
\end{enumerate}
 
%\item \label{H_PL} The function $f$ is $\gamma$-PL; namely, for every $x\in\R^d$,
%$$\norm{\nabla f(x)}{}^2\geq \gamma \left(f(x)-f_*\right).$$
% \begin{remark}
% While the Lipschitz continuity of the gradient is used all along the document, the PL property will not be required up to Section \ref{Sect: PL}.
% \end{remark}
% We denote the set of minimizers of $f$ as
% \begin{eqnarray}
% \mathcal{S} := \argmin_{x\in \R^d} f(x).
% \end{eqnarray}
The goal of this paper is to study the convergence properties of the stochastic iterative procedure described in Algorithm~\ref{eqn: derivative-free}, designed to solve numerically problem \eqref{eq:P}. 
This Algorithm is based on the finite difference approximations 
of a subset of directional derivatives of $f$ randomly chosen at each iteration. More precisely, the steps of the algorithm are summarized as follows: \\

\begin{algorithm}[H]
Let $x_0\in\R^d$, let $(\alpha_k)_{k\in\mathbb{N}}$ and $(h_k)_{k\in\mathbb{N}}$ be sequences of positive real values and let $(P_k)_{k\in \mathbb{N}}$ be a sequence of independent random  matrices in $\mathbb{R}^{d\times\ell}$ defined on the probability space $(\Omega,\mathcal{A},\mathbb{P})$. \\
\For{$k=0,1,\ldots$  }
{$x_{k+1} = x_k - \alpha_k P_k \nabla_{\left(P_k,h_k\right)}  f(x_k).$}
\caption{Zeroth order stochastic subspace algorithm}\label{eqn: derivative-free}
\end{algorithm}
 
We briefly introduce the notation and explain the main ideas behind the algorithm. Given $\ell\leq d$, a matrix $P\in\R^{d \times \ell}$ and an index $j\in \left[\ell\right]:=\left\{1,\ldots,\ell\right\}$, we let $p^{(j)}\in\R^d$ denote the $j$-th column of $P$. For $h>0$, define the vector $\nabla_{\left(P,h\right)} f(x)\in \R^{\ell}$ with entries
\begin{equation}\label{surrogate}
\left[\nabla_{\left(P,h\right)} f(x)\right]_j : = \frac{f(x+h p^{(j)})-f(x)}{h},\quad\quad \forall j\in \{1,\ldots,\ell\}. 
\end{equation}
Note that $\nabla_{\left(P,h\right)} f(x)$ is the finite difference approximation of the directional derivatives of $f$ in the directions identified by the columns of $P$. Since $f$ is differentiable, %if the columns of $P$ are normalized,
$\lim\limits_{h\to 0^+} \left[\nabla_{\left(P,h\right)}f(x)\right]_j  = \langle \nabla f(x), p^{(j)} \rangle$ for every $j\in\left[ \ell \right]$ and so 
$$\lim\limits_{h\to 0^+} \nabla_{\left(P,h\right)} f(x) = P^\top \nabla f(x).$$
Using the above notation, Algorithm \ref{eqn: derivative-free} can be re-written more explicitly  as
\begin{equation*}
x_{k+1}=x_k-\alpha_k \sum_{j=1}^{\ell} \frac{f(x_k+h_k p_k^{(j)})-f(x_k)}{h_k} \ p_k^{(j)}.
\end{equation*}
The recursion has the same structure as classical gradient descent, but the gradient at $x_k$ is computed with two different forms of  approximations. On the one hand, the directional derivatives are replaced by finite differences, no derivative of the function $f$ is required. On the other hand, only the $\ell$ directions defined by the columns of $P_k$ are used. Note that $\ell$ may be smaller than the dimension of the ambient space $d$. Moreover, the directions are chosen randomly at each iteration and are not necessarily drawn from the canonical basis. Indeed, throughout the paper the following properties are the only ones  required on the random matrices $P_k$: for every $k\in\N$,
\begin{enumerate}[label={\rm ({\bf P.\arabic{*}})}, ref={\rm\bf P.\arabic{*}}, leftmargin=4em]
		\item \label{A_Pas} $P_k^\top P_k\overset{a.s.}{=}\left(d/\ell\right)\mathbb{I}$;
			\item \label{A_PE} $\mathbb{E}P_kP_k^\top= \mathbb{I}.$
%		\item \label{A_Pprob} There exist $\varepsilon\in\left[0,1\right)$ and $\pp\in\left(0,1\right]$ such that, for every $v\in \R^d$ and every $k\geq 1$,
%			\begin{equation}
%		\mathbb{P}\left[\left(1-\varepsilon\right)\| v \|^2\leq\| P_k^\top v \|^2\right] \geq \pp;
%		\end{equation}
\end{enumerate}
Next, we discuss several examples of algorithms that can be derived as special cases, see also Section \ref{sect: special cases}. 
%In section \ref{sect: special cases} we provide several examples of matrices that satisfy these properties., here we mention two.
\begin{example}[Coordinate descent]
If $P_k$ contains $\ell$ columns of $\I_d$ chosen uniformly at random without replacement and with random sign, say $\pm e_{k_1}, \ldots \pm e_{k_\ell}$, then $p_k^{(j)} = \pm e_{k_j}$ for $j=1,\ldots,\ell$, 
and  Algorithm~\ref{eqn: derivative-free} corresponds to a discretized version of parallel block-coordinate descent.
% Thus, 
%% $$
%% P_k \nabla_{(P_k,h)}f(x_k) = (e_{k_1} \ldots e_{k_\ell})\begin{pmatrix} \frac{f(x_k + e_{k_1}) - f(x_k)}{h} \\ \vdots \\ \frac{f(x_k + e_{k_\ell}) - f(x_k)}{h}
%% \end{pmatrix}
%% $$
%% and
%% $$
%% P_kP_k^\top \nabla f(x_k) = (e_{k_1} \cdots e_{k_\ell}) \begin{pmatrix}
%% e_{k_1}^\top \\ \vdots \\ e_{k_\ell}^\top 
%% \end{pmatrix} \nabla f(x_k).
%% $$
%% That is,
%Theorems \ref{Th:main1} and \ref{Th:main2} provide convergence results for randomized parallel block-coordinate descent and its discretized counterpart.
\end{example}
\begin{example}[Spherical smoothing]
Consider $P_k = (\sqrt{d/\ell})\,Q_k\I_{d\times\ell}$, where $Q_k$ is as in the $QR$-decomposition of a matrix $Z_k = Q_kR_k \in \reals^{d \times d}$  with $R_{k_{ii}} >0,$ and each element of $Z_k$ is drawn independently from $\mathcal{N}(0,1)$.
The matrix $\I_{d \times \ell}$ truncates $Q_k$ to its first $\ell$ columns so $Q_k\I_{d\times\ell}$ corresponds to $\ell$ columns of the random orthogonal matrix distributed according to the Haar measure on orthogonal matrices \cite{Mezzadri}. 
In other words, the columns $p_k^{(j)}$ are orthogonal and distributed uniformly on the sphere for all $j$. Then, when $\ell=1$, $P_k\nabla_{(P_k,h)}f(x_k)$ is a spherical smoothing estimate of the gradient, as described in, e.g., \cite{MR2298287,berahas2019theoretical}. For the case $\ell >1$ it is more common  to sample $p^{(j)}$ independently and uniformly on the sphere \cite{berahas2019theoretical}, but in our case, to satisfy Assumptions \eqref{A_Pas} and \eqref{A_PE}, the columns of $P$ must be orthonormal, similar to \cite{Kozak2021Stochastic}. 
The advantage of a matrix $P$ with orthonormal columns is discussed at length in \cite{Kozak2021Stochastic}. We remark here  that this property is a valid extension of traditional spherical smoothing to subspaces of dimension greater than one, and is required to obtain our results and connect Algorithm \eqref{eqn: derivative-free} with discrete gradient descent with $\ell=d$.
\end{example}
\begin{example}[Gradient descent on random subspaces]  
If $h_k =h$ for every $k\in\mathbb{N}$ and $h \to 0$, then $\nabla_{(P_k,h_k)}f(x_k)$ becomes $P_k^\top\nabla f(x_k)$ and the algorithm reduces to 
\begin{equation}\label{eqn: gradient-free}
x_{k+1} = x_k - \alpha_k P_k P_k^\top \nabla f(x_k).
\end{equation}
 For shorthand, when we reference Algorithm~\ref{eqn: derivative-free} with $h_k=0$, we are referring to the use of exact directional derivatives as in recursion~\ref{eqn: gradient-free}, which has been introduced and studied in \cite{Kozak2021Stochastic}. Our analysis allows to recover and improve on previous results for the iteration in~\eqref{eqn: gradient-free}.
 \end{example}
 \begin{example}[Gradient descent]  When $\ell=d$ we have $PP^\top = P^\top P = \I_d$. Hence,  $P$ is an orthonormal basis for $\reals^d$, so that  discrete gradient descent and gradient descent are recovered as special cases of Algorithm \ref{eqn: derivative-free} and recursion \ref{eqn: gradient-free} respectively.
\end{example}
Before stating and discussing our main results, we add one remark. 
\begin{remark}[Derivative-free optimization and automatic differentiation]
 From a practical point of view, the implementation of methods based on exact gradient computations, such as~\eqref{eqn: gradient-free},  is restricted to cases where exact directional derivatives are available. Under certain conditions, one can use automatic differentiation to obtain directional derivatives, however this requirement restricts the user to particular software, and precludes experiments wherein the function is accessed only via blackbox function evaluations, as is more common in derivative-free optimization.
 Importantly, the recursion based on exact gradient computations cannot be used when the simulations are physical (such as robotics, and many engineering examples), or when the objective can not be described by the elementary functions available to automatic differentiation software (such as in reinforcement learning).  
% One motivation of this paper is to provide convergence guarantees for the zeroth-order recursion of Algorithm \eqref{eqn: derivative-free}, which generalizes \eqref{eqn: gradient-free} but is simple to implement provided the objective function can be evaluated. 
 \end{remark}
\section{Main results}\label{sec:main}
%The theorems in the rest of the document are sharper and more detailed, but the main practical results are the following. Assume the conditions
%\begin{itemize}
%	\item[(\ref{H_existence}) \ ] $\mathcal{S}\neq \emptyset$;
%	\item[(\ref{H_lip}) \  ] $\nabla f$ Lipschitz continuous;
%	\item[(\ref{A_Pas}) \ ] $P_k^\top P_k=\left(d/\ell\right)\mathbb{I}$ a.s.; and
%	\item[(\ref{A_PE}) \ ] $\mathbb{E}_k(P_kP_k^\top)= \mathbb{I}$.
%\end{itemize}
In this section,  we summarize the main results of the paper. 
We provide convergence results, explain the dependence on the discretization parameters $h_k$ and the choice of the stepsizes $\alpha_k$, and provide context for the results within the larger body of literature. The section contains two theorems. Theorem \ref{Th:main1} establishes convergence properties of Algorithm~\ref{eqn: derivative-free} in the convex case, while Theorem \ref{Th:main2} deals with objective functions satisfying the Polyak-Łojasiewicz condition \ref{H_PL} without requiring convexity. The theorems are simplified versions of the results in Sections \ref{Sect: Convex case} and \ref{Sect: PL}, where more detailed statements and proofs can be found.

\begin{theorem}[Convergence - convex case]\label{Th:main1}  Assume that  conditions \ref{H_lip} and \ref{H_existence} are satisfied and suppose that  \ref{A_Pas} and \ref{A_PE} hold. Let $\seq{x_k}$ be a random sequence generated by Algorithm~\ref{eqn: derivative-free}. Set $\Lambda:=\lambda d/\ell$.\\
\begin{itemize}
\item[(i) ] Let $0\leq h_k\leq \hup$ and $\alpha_k=\alpha$ for some $\hup > 0$ and $\alpha\in \left(0,1/\Lambda\right)$. Then, for explicit constants $C_1, C_2, C_3$ and $C_4$,
$$\min_{j\in[k]}\mathbb{E}(f(x_j)-f_*) \quad \leq \quad  \max\{C_1/(k+1)+C_2\hup^2+C_3\hup, \ C_2 \hup^2+C_4\hup\}.$$
\item[(ii) ] Set $h_k=1/k^r$ with $r>0$ and $\alpha_k=\alpha/k$ with $\alpha\in \left(0,1/\Lambda\right)$. Then,
\begin{equation*}
		\lim_k f(x_k) \quad  \eas \quad  f_*   \ \ 
			\end{equation*}
		and, for an explicit constant $D_0>0$,
	\begin{equation*}
	\begin{split}
	\min_{j\in[k]} \mathbb{E}\left[f(x_j) -f_*\right] \quad  \leq \quad D_0/\ln k.
	\end{split}
	\end{equation*}
	\item[(iii) ] Set $h_k=1/k^r$ with $r>1$ and $\alpha_k=\alpha$ with $\alpha\in \left(0,1/\Lambda\right)$. Then,
	\begin{equation*}
	\lim_k f(x_k) \quad  \eas \quad  f_* \ \
	\end{equation*}
	and, for an explicit constant $D>0$,
	\begin{equation*}
	\begin{split}
	\min_{j\in[k]} \mathbb{E}\left[f(x_j) -f_*\right]  \quad  \leq \quad D/k.
	\end{split}
	\end{equation*}
	Moreover, there exists a random variable $x_*$ with values in $\argmin f$ such that $x_k\toas x_*$.\\
	\item[(iv) ] Consider recursion ~\ref{eqn: gradient-free} with $\alpha_k=\alpha$ and $\alpha\in \left(0,2/\Lambda\right)$. Then
	\begin{equation*}
	f(x_k) - f_* \quad  \eas \quad o\left(1/k\right) \ \
	\end{equation*}
	and there exists a random variable $x_*$ with values in $\argmin f$ such that $x_k\toas x_*$.
	\end{itemize}
\end{theorem}
\begin{proof}
The results in $(i), (ii), (iii)$ and $(iv)$ are special cases of Theorems \ref{th:convex_inexact}, \ref{th:FejerWeak}, \ref{th:FejerStrong} and \ref{th:FejerSStrong} - respectively.
\end{proof}
More detailed results - and under milder assumptions - are given in the theorems cited above. There, for instance, one can find explicit computations of the constants $C_1, C_2, C_3, C_4$ and $D$. Some comments on Theorem \ref{Th:main1} are in order. In $(i)$ we show the case in which the discretization $h_k$ does not decrease to zero. We get a stability estimate for the expectation of the function values at the best iterate, depending on the upper-bound on the error noise $\hup$. Notice that we recover the rate $C_1/(k+1)$ as $\hup\to 0$. On the other hand, for $\hup>0$, running the algorithm beyond a  certain number of iterations does not lead to any guarantee of improvement to the best iterate. More precisely, if we consider a stopping time proportional to $1/\hup$, then overall accuracy will be $O(\hup)$ (assuming $\hup\leq 1$ for simplicity). In the rest of Theorem \ref{Th:main1} we  require $h_k\to 0$. In $(ii)$ the discretization sequence is allowed to converge to zero (polynomially but) arbitrarily slowly, but also a vanishing stepsize $\alpha_k$ is required and the results we can obtain are quite weak: the function values converge to the optimum with a logarithmic rate (in expectation and for the best iterate). Better results can be obtained if $h_k$ converges to zero fast enough, as shown in $(iii)$. In this case we gain on two different sides: first, a constant stepsize $\alpha$ can be used, which is convenient from a numerical point of view; second, we get faster convergence rates for the function values of the form $1/k$ (again in expectation and for the best iterate); third, we have a.s. convergence of the iterates. To the best of our knowledge, this is the first result showing convergence of the iterates for these types of zeroth-order methods in the general convex case; recall that special cases of our method include several well-known methods such as coordinate descent and smoothing on a sphere. Finally, for  recursion~\eqref{eqn: gradient-free} in $(iv)$, we obtain an a.s. convergence rate asymptotically faster than $1/k$ for the last iterate (and not only for the best one).  We add four remarks. 

\begin{remark}\label{LipschitzConst}
    In the setting of Theorem \ref{Th:main1}, $\Lambda=\lambda d / \ell$ plays the role of the Lipschitz constant. In $(iv)$ the choice of stepsize is bounded above by the classical limit $2/\Lambda$. On the other hand, as will be evident in the proofs, $(i)$-$(iii)$ require the stepsize to be bounded above by $1/\Lambda$.
\end{remark}
\begin{remark}
The ergodic iterate $\bar{x}_k:=\left(\sum_{j=1}^{k}\alpha_j\right)^{-1}\sum_{j=1}^{k} \alpha_j x_j$ attains the same rates as the ones above for the best iterate, but it is of little practical interest. Indeed, in order to apply the algorithm, $f$ has to be evaluated at each iterate $x_k$ and so it is possible to just keep the one that achieves the minimal value function.
\end{remark}
\begin{remark}
Let $N$ denote the number of function evaluations required to perform $k$ iterations. At each iteration Algorithm \ref{eqn: derivative-free} uses \ $\ell+1$ \ function evaluations. Then $N=(\ell+1)k$ and the rates in Theorem \ref{Th:main1} $(i)$-$(iii)$ can be easily re-written in terms of function evaluations. The same observation holds for the results in Theorem \ref{Th:main2}.
\end{remark}
\begin{remark}[Comparison with previous work]
Here we compare to the results in the literature dealing with a convex objective function. 
\begin{itemize}
\item The case $\ell=1$ with Gaussian sampling is considered in \cite{nesterov2011random} where one can find results similar to ours in terms of convergence rates but with worse constants.  See  Remark \ref{rmk:Nest_Comp} for a precise comparison.
\item In \cite{duchi2015optimal} a mirror-descent variant of Algorithm \ref{eqn: derivative-free} is proposed to deal with a stochastic objective, similar to the setting in  \cite{MR50243}.  A rate $O(1/\sqrt{k})$  in expectation is derived, however the results should not be compared directly to our related results in $(iii)$, as the setting is more challenging. % We do not conjecture how the rate of \cite{duchi2015optimal} would be affected for a fixed objective as in our case. 
See Remark \ref{rmk:Duchi_Comp}. 
\item In \cite{Kozak2021Stochastic} only the setting $(iv)$ is studied. Stricter assumptions are required (e.g., existence of a finite $R$ such that $$\max_{x_*}\max_x \{\norm{x-x_*} : f(x) \leq f(x_0)  \} \leq R,$$ and only convergence in expectation is provided. The special case of parallel coordinate descent has been considered in \cite{tappenden2018complexity} and  \cite{salzo2019parallel}.
% \item The same limitations as \cite{Kozak2021Stochastic} are present in \cite{MR3347548}, with the additional restriction that \cite{MR3347548} deals only with coordinate descent.
\end{itemize}
\end{remark}
The  following Corollary considers the case where the algorithm is run for a finite number of iterations known a priori. 
\begin{corollary}\label{cor:rate} Under the assumptions of Theorem \ref{Th:main1} $(i)$, let $K\in\mathbb{N}$ and  $\hup \leq 1/K$. Then, for some constant $C$,
\begin{align}\label{eq:cor1Result1}
    \min_{j\in[K]}\mathbb{E}(f(x_j)-f_*) \leq  C/K.
    \end{align}
\end{corollary}
In particular, given a tolerance $\varepsilon>0$, it is possible to choose a number of iteration $K\in\mathbb{N}$ such that $C/K\leq \varepsilon$ and a discretization $\hup \leq 1/K$, in order to get
\begin{align}\label{eq:cor1Result1}
    \min_{j\in[K]}\mathbb{E}(f(x_j)-f_*) \leq  \varepsilon.
    \end{align}
    \ \\
    
Next we state and discuss a second set of results derived under different assumptions on the objective function. 
It is well-known that first order methods exhibit favorable convergence results for strongly convex functions. More recently, it has been proved that improved convergence rates can also be obtained when the objective function satisfies weaker geometrical assumptions that do not require convexity \cite{attouch2013convergence,karimi2016linear}. In this paper we consider the (global) Polyak-Łojasiewicz condition: 

\begin{enumerate}[label={\rm ({\bf H.\arabic{*}})}, ref={\rm\bf H.\arabic{*}}, leftmargin=4em]\setcounter{enumi}{2}
\item \label{H_PL} the function $f$ is $\gamma$-PL; namely, for every $x\in\R^d$,
$$\norm{\nabla f(x)}{}^2\geq \gamma \left(f(x)-f_*\right).$$
\end{enumerate}
The main example of functions satisfying the global PL condition is the class of strongly convex functions. For more examples in the non-convex setting, see \cite{attouch2013convergence} and the numerical experiments in Section \ref{sec:num}.
We stress again that in the next result we assume \ref{H_PL} but not \ref{H_existence}; that is, we do not require convexity of the objective function. % On the other hand, the existence of minimizer is a consequence of \ref{H_PL}, see \cite{}. 
\begin{theorem}[Convergence - PL case]\label{Th:main2} Assume that  conditions \ref{H_lip} and \ref{H_PL} are satisfied and suppose that  \ref{A_Pas} and \ref{A_PE} hold. Let $\seq{x_k}$ be a random sequence generated by Algorithm~\ref{eqn: derivative-free}. Set $\Lambda:=\lambda d/\ell$ and let $\alpha\in \left(0,2/\Lambda\right)$. Fix a constant $w\leq1$ such that $0<\w<2-\Lambda \aup$, and define $\eta:=1-\w\alpha\gamma/2$.\\
\begin{enumerate}
\item[(i') ] Set $h_k \leq \hup$  for some $\hup>0$ and $\alpha_k=\alpha$. Then, for an explicit constant $C_1>0$,
	\begin{equation*}
	\begin{split}
	\mathbb{E}\left[f(x_k)-f_*\right] & \quad  \leq \quad  \eta^k \left(f(x_0)-f_*\right)  +C_1 \ \hup^2\left[1-\eta^k\right].
	\end{split}
	\end{equation*} 
\item[(ii') ] Set $h_k=1/k^r$ with $r>0$ and $\alpha_k=\alpha$. Then, there exists a constant $C_2>0$ such that
$$\mathbb{E}\left[f(x_k)-f_*\right]\quad  \leq \quad  C_2/k^{2r}.$$
\item[(iii') ] Set $h_k=\sqrt{\eta^k/k^r}$ with $r>1$ and $\alpha_k=\alpha$. Then, for an explicit constant $C_2>0$,
$$\mathbb{E}\left[f(x_k)-f_*\right]\quad \leq \quad  C_2\eta^k.$$
\item[(iv') ] Consider recursion~\eqref{eqn: gradient-free} with $\alpha_k=\alpha$. Then,
	\begin{equation*}
	\begin{split}
	\mathbb{E}\left[f(x_k)-f_*\right] &\quad  \leq \quad \eta^k \left(f(x_0)-f_*\right).
	\end{split}
	\end{equation*} 
	\end{enumerate}
\end{theorem}
\begin{proof}
The results in $(i'), (ii')$ and $(iii')$ are special cases of Theorems \ref{th:PL_lin}, \ref{th:PL_alphaconst} and \ref{th:linrate} - respectively; while the result in $(iv')$ is presented here for completeness but is already considered in \cite{Kozak2021Stochastic} - see Corollary $2.3$.
\end{proof}
	Sharper and more detailed results are given in the theorems cited above, where the reader can also find the explicit computations of the constants involved in the rates. We remark only that, for a constant stepsize $\alpha_k=1/\Lambda$, the decreasing rate in $(i')$, \ $(ii')$ and $(iii')$ is given by
$\eta=1-\gamma/(2\Lambda).$
As a general comparison with Theorem \ref{Th:main1}, note that the results in Theorem \ref{Th:main2} do not involve the best iterate but only the last one. As in the first result of Theorem \ref{Th:main1}, in Theorem \ref{Th:main2} 
$(i')$ the parameter $h_k$ does not necessarily vanish and so the error produced by the finite difference discretization of the gradient does not converge to zero. This explains the substantial difference between this result and the ones in $(ii')$-$(iv')$: in the upper-bound of $(i')$, for $k\to+\infty$ the right-hand side does not vanish. We can only guarantee that the expectation of the function evaluations converges with a linear rate to a sublevel set of $f$ with value $f_*+C_1 \hup^2$, also called \emph{error dominated region}. On the other hand note that, in comparison with Theorem \ref{Th:main1} $(i)$, the upper-bound does not diverge with the iterations but it remains bounded. In $(iii')$, we study the case of polynomial decay of the discretization parameter; namely, a decay of the form $h_k=1/k^r$ with $r>0$. In this case, the upper-bound is proportional to $1/k^{2r}$ and  the rate gets better for a faster decay of the discretization parameter. In $(iii')$, under the assumption of a sufficiently fast (exponential) decay of $h_k$ and with a constant stepsize $\alpha_k$, we get a linear rate of convergence to the optimal value. Finally in $(iv')$, for the recursion in \eqref{eqn: gradient-free} with $\alpha_k$ constant, we recover the linear rate already shown in \cite{Kozak2021Stochastic}. % Again for simplicity, we did not present here the additional result stated in Theorem \ref{th:Chung}, where we study the case of polynomial decay of the discretization parameter; namely, a decay of the form $h_k=1/k^r$ with $r>0$. In this case, the upper-bound is proportional to $1/k^{2r}$ and so the rate is better for a faster decay of the discretization parameter.  
% Assuming a fast decay of $h_k$, we get a linear rate of convergence to the optimum in $(iii')$. Finally, in $(iv')$, for recursion \eqref{eqn: gradient-free} $\alpha_k$ constant, we recover the linear rate already shown in \cite{Kozak2021Stochastic}.
We add three remarks.
\begin{remark}[Stepsize bounds]
    Following the same discussion as in Remark \ref{LipschitzConst}, in all the results of Theorem \ref{Th:main2} the stepsize must be bounded above by the classical limit $2/\Lambda$.
\end{remark}
\begin{remark}[Adaptivity]
An important consequence of the analysis in the previous theorems is that Algorithm~\ref{eqn: derivative-free} is \emph{adaptive}; that is, knowledge of the specific properties of $f$ is not needed to ensure the corresponding convergence results. We give an explicit example. Consider an objective for which assumptions \ref{H_lip}, \ref{H_existence} are satisfied and we run the algorithm with $h_k=1/k^r$ for $r>1$ and $\alpha_k=\alpha$ for $\alpha\in \left(0,\ell/(\lambda d)\right)$. Then the results in Theorem \ref{Th:main1} $(iii)$ hold; namely, we have convergence of the iterates, almost sure convergence of the function values and a $1/k$ rate for the best iterate in expectation. In the same setting, if the function satisfies the PL inequality ($\ref{H_PL}$), from Theorem \ref{Th:main2} $(ii')$ we get automatically a $1/k^{2r}$ rate for the last iterate in expectation. 
\end{remark}

\begin{remark}[Almost sure convergence]
Under the assumptions of Theorem \ref{Th:main2} $(ii')$, we have that\\
$\mathbb{E}\left[f(x_k)-f_*\right]  \lesssim C_2 / {k^{2r}}$. 	Then,  If $r>1/2$ the right hand side is summable and, by Markov's inequality and Borel-Cantelli Lemma, $f(x_k)\toas f_*$, so that $d(x_k)\toas 0.$
	In particular, if the function $f$ has a unique minimizer $x_*$, \ $x_k\toas x_*$. Analogous reasoning holds for Theorem \ref{Th:main2} $(iii')$ and $(iv')$.
\end{remark}

Finally,  the  following Corollary considers the case where the algorithm is run for a finite number of iterations known a priori. 
\begin{corollary}\label{cor:rate} Under the assumptions of Theorem \ref{Th:main2} $(i')$, let $K\in\mathbb{N}$ and  $\hup \leq\eta^{K/2}$. Then,
\begin{align}\label{eq:cor1Result1}
    \Expectation \left[f(x_{K}) - f_*\right]  &\leq \left(f(x_0) - f_* +C_1\right) \eta^K.
    \end{align}
    Moreover, if  $d(\cdot)$ is  the Euclidean distance of the argument to the set $\argmin_x f$, then
\begin{align}\label{eq:cor1Result1}
\nonumber    \Expectation \ d(x_{K}) &\leq \frac{2}{\gamma}\left(f(x_0) - f_* + C_1\right) \eta^K .
    \end{align}
\end{corollary}
\begin{proof}
The first result follows directly from the assumption $\hup \leq\eta^{K/2}$ and Theorem \ref{Th:main2} $(i')$. The second claim follows from the previous one and the following inequality, that holds  under \ref{H_PL} - see e.g. \cite{bolte2017error}: for every $x\in \R^d$,
$ d(x) \leq 2/{\gamma}\left(f(x) - f_*\right).$
\end{proof}
Equation \eqref{eq:cor1Result1} allows for a solution minimizing $f$ up to an arbitrary desired accuracy $\varepsilon>0$ in expectation. To this aim it is sufficient to choose
\[
K=\frac{\ln\Big(\varepsilon/(f(x_0) - f_* +C_1)\Big)}{\ln \eta}\quad \text{ and }\quad \hup\leq \eta^{K/2}.
\]
Note that the dependence on the number of iterations is logarithmic in $\varepsilon$.  A similar observation holds for $ \Expectation  d(x_{K})$. If $f$ has a unique minimizer $x_*$ (e.g. $f$ is  strongly convex) Corollary \ref{cor:rate} provides a   convergence  rate for the iterates to the minimizer. 

In the next sections we provide more detailed statements and proofs of the results. We start with some useful preliminary results that are the basis for the development in later sections.

\section{Preliminaries}\label{sec:pre}
\subsection{Notation and stepsize assumptions}
For $k\in \N$, we define \\
$$f_k:=f(x_k), \quad \nabla f_k:= \nabla f(x_k), \quad \nabla_k f_k:= \nabla_{\left(P_k,h_k\right)} f(x_k).$$
 
Now we provide the main conditions on the parameters of the proposed algorithm. Setting $\Lambda:=\lambda d/\ell$, the assumptions we consider on the sequences of stepsizes $\seq{\alpha_k}$ and discretizations $\seq{h_k}$ are as follows:\\
\begin{enumerate}[label={\rm ({\bf A.\arabic{*}})}, ref={\rm\bf A.\arabic{*}}, leftmargin=4em]
	\item \label{A_stepwith2}  $0 < \alpha_k \leq \aup<2/\Lambda$;
	\item \label{A_stepwith1} 
	$0 < \alpha_k \leq \aup<1/\Lambda$;
	\item \label{A_stepnotinell1} $(\alpha_k)\notin \ell^1$ and $(h_k)$ is bounded above by some $\hup\geq 0$;
	\item \label{A_stepbb} $(\alpha_k)$ is bounded below by some $\alow>0$ and $(h_k)$ is bounded above by some $\hup\geq 0$.
\end{enumerate}
%\begin{remark}
	Note that \ref{A_stepwith1} implies \ref{A_stepwith2} while \ref{A_stepbb} implies \ref{A_stepnotinell1}.
%\end{remark}
%\begin{remark} 
Also note that $\Lambda$ plays a role analogous to the Lipschitz constant in  classical gradient descent, and  $\Lambda=\lambda$ when $\ell=d$.
%\end{remark}

% \begin{remark}[Comparison with previous work]
% Here we provide discussion of other methods in the literature that analyze results when the objective function satisfies the Polyak-Łojasiewicz inequality or is strongly convex. 
% \begin{itemize} 
% % \item In \cite[Sect.5]{nesterov2017random}, Gaussian smoothing is proposed and analyzed for strongly convex functions. Only the $\ell=1$ case is discussed. The results are comparable to $(i')$, $(iii')$ and $(iv')$, with a linear rate of convergence that has worse constants than our method. % It is claimed but not shown in \cite{nesterov2017random} that a decaying $\alpha_k$ and $h_k$ will result in an asymptotically convergent sequence. 
% \item In \cite{Kozak2021Stochastic} the computation of the directional derivatives is exact and the stepsize is constant, as in $(iv')$, and our results are identical. They do not discuss the settings of $(i')$, $(ii')$ or $(iii')$.
% \item  In \cite{MR3347548} only the special case of coordinate descent is analyzed. The computation of the directional derivatives is exact and the stepsize is constant, as in $iii')$, and our results are identical. They do not discuss the settings of $i')$ or $ii')$.
% \end{itemize}
% \end{remark}
\subsection{Auxiliary bound}
We start the analysis with an auxiliary lemma that estimates the distance between the surrogate of the gradient used in the algorithm and the projected exact gradient. For this first result we assume only Lipschitz continuity of the gradient and condition \ref{A_Pas} on the matrix $P_k$; in particular, convexity of $f$ is not needed. The upper-bound is a simple consequence of the Descent Lemma \ref{DescentLemm} and  the proof  is in  the appendix.  
\begin{lemma}\label{lem} Let $f$ be a function satisfying \ref{H_lip} and $P_k\in \R^{d \times \ell}$ a random matrix satisfying \ref{A_Pas}. Then, for every $x\in\R^d$ and every $k\in\N$,
	$$\norm{\nabla_{\left(P_k, h_k\right)} f(x)-P_k^\top \nabla f(x)}{}\leq\frac{\lambda d h_k }{2\sqrt{\ell}}  \ \ \ \ \ \ \ a.s.$$
\end{lemma}

\begin{remark}[Discussion on the bias] Consider a sufficiently regular function $f: \R^d \to \R$. Fix $h>0$, $\ell=1$ and a $\R^d$-valued random vector $p$ satisfying Assumptions \ref{A_Pas}-\ref{A_PE}; namely,
$$\norm{p}\eas \sqrt{d} \quad \text{and} \quad \mathbb{E}\left[pp^\top\right]=\mathbb{I}.$$
For every $x\in\R^d$, by Taylor expansion we have that
\begin{eqnarray}\label{remkkk}
%\begin{split}
   && \mathbb{E} \left[\nabla_{(p,h)}f(x)\right] =\mathbb{E}\left[\frac{f(x+hp)-f(x)}{h} \ p\right]\\
    &=&\nabla f(x)+\frac{h}{2} \ \mathbb{E}\left[\langle \nabla^2 f(x) p,p\rangle \ p\right]+\frac{h^2}{6} \ \mathbb{E}\left[\langle \left[\nabla^3 f(x)p\right] p,p\rangle \ p\right]+o(h^2).\nonumber
%\end{split}
\end{eqnarray}
We  focus on  two fundamental choices for the distribution of $p$.\\
% In the convex quadratic case, namely for $f(x)=\langle Ax,x\rangle + \langle b,x\rangle$ with $A$ symmetric and definite positive, the Hessian is constant and given by $A$. 
\begin{itemize}
    \item In the coordinate-wise framework, $p$ is distributed uniformly on the (discrete) set $(\sqrt{d}\  e_i)_{i=1}^d \ \cup \ (-\sqrt{d}\  e_i)_{i=1}^d$, where $e_i$ represents the $i$-th vector of the canonical basis. Then the first-order term in \eqref{remkkk} is zero:
\begin{eqnarray*}
%\begin{split}
\mathbb{E}\left[\langle \nabla^2 f(x) p,p\rangle \ p\right]&=& \frac{1}{2d}\sum_{i=1}^d \langle \nabla^2 f(x) \sqrt{d} \ e_i,\sqrt{d} \ e_i\rangle \ \sqrt{d} \ e_i\\
&-&\frac{1}{2d}\sum_{i=1}^d \langle \nabla^2 f(x) \sqrt{d} \ e_i,\sqrt{d} \ e_i\rangle \ \sqrt{d} \ e_i=0.
%\end{split}
\end{eqnarray*}
\item The same holds for spherical smoothing, where $p$ is distributed uniformly on the (continuous) set given by the sphere of radius $\sqrt{d}$ and centered at the origin, that we denote by $\partial B$. Also in this case, indeed,
\begin{equation*}
\begin{split}
    \mathbb{E}\left[\langle \nabla^2 f(x) p,p\rangle \ p\right]=\int_{\partial B} \langle \nabla^2 f(x) p,p\rangle \ p \ dp = 0,
\end{split}
\end{equation*}
as we are integrating an odd function on a domain that is symmetric with respect to the origin.
\end{itemize}
% So, for spherical smoothing, the proposed surrogate gradient is unbiased up to the first-order term in $h$; while the latter does not hold for the coordinate-wise framework. This property partially explains why spherical smoothing has a better behaviour than coordinate-wise procedures. 
So, in the sufficiently regular case analysed above, the surrogate gradient proposed by coordinate-wise and spherical smoothing is unbiased up to the first-order term in $h$. But notice also that, in both cases, the term in $h^2$ is non-zero in general and the surrogate gradient is biased. This excludes the direct applicability of many theoretical studies for general stochastic algorithms such as \cite{bertsekas2000gradient} and motivates the introduction of Lemma \ref{DescentLemm}, in which we bound a.s. the distance between the surrogate gradient and the projection of the exact one. On the other hand, the case of quadratic functions - in which both coordinate-wise and spherical smoothing are  unbiased -  deserves a  tailored convergence analysis that we leave for future work. For  refined properties of spherical smoothing and coordinate-wise descent, see \cite{berahas2019theoretical}  and Lemma 1, Remark 1 in \cite{Kozak2021Stochastic}.\\
\end{remark}

\subsection{A quasi-descent lemma}\label{sec:as}
In the next proposition and corollary we obtain an a.s. estimate for the decrease of the objective function values. This result is the fundamental tool for the analysis in Section \ref{Sect: PL}, and it is also of standalone interest.
\begin{proposition}\label{prop:as} 
Assume \ref{H_lip}, \ref{A_Pas} and \ref{A_stepwith2} and let the random sequence $\seq{x_k}$ be generated by Algorithm~\ref{eqn: derivative-free}. Fix a constant $w$ such that $0<\w<2-\Lambda \aup$ \  and $ \w\leq 1$. Define
%	\begin{equation*}
%	\begin{split}
$	C  := {(\ell \Lambda^2)}/{(8\min\left(1, 2-\Lambda\aup-\w\right))}.$
%$	C  := \frac{\ell \Lambda^2}{8\min\left(1, 2-\Lambda\aup-\w\right)}.$
%	\end{split}
%	\end{equation*}
	Then, for every $k\in\mathbb{N}$,
	\begin{equation*}
	\begin{split}
	f_{k+1}-f_k & \leq  -\frac{\w \alpha_k}{2} \| P_k^\top\nabla f_k \|^2  + C \alpha_k h_k^2  \ \ \ \ \ \ \ a.s.
	\end{split}
	\end{equation*}
	In particular, $f_{k+1}-f_k  \leq C \alpha_k h_k^2$ a.s.
\end{proposition}
 
Before showing the proof of Proposition \ref{prop:as}, we make some comments on the assumptions and conclusion and provide an important consequence under additional assumptions. First note that Proposition\ref{prop:as} requires neither \ref{A_PE} nor \ref{H_existence}-\ref{H_PL}. %\textcolor{red}{; and that the consequent holds almost-surely}.
The difference $f_{k+1}-f_k$ is bounded a.s. by the sum of two terms: a negative term depending on the squared norm of the projected gradient $P_k^\top\nabla f_k$ and an error term proportional to $\alpha_k h_k^2$. 
% In particular, with recursion \eqref{eqn: gradient-free} the objective function values are a.s. decreasing (this fact is used in the proof of Theorem \ref{th:FejerSStrong}). 
\begin{remark}\label{rem:h0}
	Notice that the term mutiplying $\| P_k^\top\nabla f_k \|^2$ does not depend on $h_k$. A result analogous to Proposition~\ref{prop:as} can be obtained for recursion \eqref{eqn: gradient-free}, corresponding to the limit $h\to 0$. In this case, for every $k\in\N$,
\[
f_{k+1}-f_k \leq -\frac{\w\alpha_k}{2}  \|P_k^\top\nabla f_k\|^2   \ \ \ \ \ \ \ a.s.
\]
In particular, the objective function values are a.s. decreasing; namely, $f_{k+1}\leq f_k$ a.s. This fact is used later, in the proof of Theorem \ref{th:FejerSStrong}.
\end{remark}

Finally, adding condition \ref{A_stepbb} to the hypothesis of Proposition~\ref{prop:as}, we get the following corollary.
\begin{corollary}\label{cor:as}
	Assume \ref{H_lip}, \ref{A_Pas}, \ref{A_stepwith2} and \ref{A_stepbb}. Let $\seq{x_k}$ be a random sequence generated by Algorithm~\ref{eqn: derivative-free}. Let $0<\w<2-\Lambda \aup$ \  and $ \w\leq 1$. Then, for every $k\in\mathbb{N}$,
	\begin{equation}\label{estt}
	\begin{split}
	f_{k+1}-f_k & \leq  -\frac{ \w \alow}{2} \| P_k^\top\nabla f_k \|^2  + C \aup \hup^2  \ \ \ \ \ \ \ a.s.
	\end{split}
	\end{equation}
\end{corollary}
\begin{proof}[Proof of Proposition \ref{prop:as}]
	For notational simplicity, we denote $P_k$ by $P$ and set $e_k:= \nabla_k f_k-P_k^\top\nabla f_k$. By Assumption \ref{A_Pas}, for every $v\in\R^{\ell}$,
\begin{equation}\label{eqfff}
\norm{Pv}{}^2\eas\frac{d}{\ell}\norm{v}{}^2.
\end{equation}
By Hypothesis \ref{H_lip} combined with Descent Lemma \ref{DescentLemm} and Algorithm~\ref{eqn: derivative-free}, we obtain
\begin{equation}\label{desc}
\begin{split}
f_{k+1}-f_k & \leq \langle \nabla f_k, x_{k+1}-x_k\rangle + \frac{\lambda}{2}\| x_{k+1}-x_k\|^2\\
&= -\alpha_k \langle P^\top\nabla f_k, \nabla_k f_k\rangle + \frac{\alpha_k^2\lambda}{2}\| P\nabla_k f_k\|^2.
\end{split}
\end{equation}
For the last term,  we add and subtract $PP^{\top}\nabla f_k$ and recall the definition of $e_k$, to get
\begin{equation*}
\begin{split}
\frac{\alpha_k^2\lambda}{2}\| P\nabla_k f_k\|^2
& = \frac{\alpha_k^2\lambda}{2}\| P\nabla_k f_k-PP^\top\nabla f_k+PP^\top\nabla f_k\|^2\\
& = \frac{\alpha_k^2\lambda}{2}\| Pe_k \|^2+ \frac{\alpha_k^2\lambda}{2}\| PP^\top\nabla f_k\|^2+\alpha_k^2\lambda \langle Pe_k, PP^\top\nabla f_k\rangle.
\end{split}
\end{equation*}
Now, using again the algorithm and relation \eqref{eqfff} in inequality \eqref{desc}, we have
\begin{eqnarray*}
\begin{split}
f_{k+1}-f_k  
&\leq
 -\alpha_k \langle P^\top\nabla f_k, \nabla_k f_k\rangle  \\&
 + \frac{\alpha_k^2\lambda d}{2\ell}\| e_k\|^2+ \frac{\alpha_k^2\lambda d}{2 \ell}\| P^\top\nabla f_k\|^2+\frac{\alpha_k^2\lambda d}{\ell} \langle \nabla_k f_k-P^\top\nabla f_k, P^\top\nabla f_k\rangle \\
& = \alpha_k \left( \Lambda\alpha_k-1 \right) \langle P^\top\nabla f_k,  \nabla_k f_k\rangle  + \frac{\Lambda\alpha_k^2}{2}\| e_k\|^2 - \frac{\Lambda\alpha_k^2}{2}\| P^\top\nabla f_k\|^2,
\end{split}
\end{eqnarray*}
where we recall that $\Lambda$ is defined as $\Lambda:=\lambda d/\ell$. Write $\nabla_k f_k=e_k+P^\top \nabla f_k$, to get
\begin{equation*}
\begin{split}
f_{k+1}-f_k  & \leq \alpha_k \left(\Lambda\alpha_k -1 \right) \langle P^\top\nabla f_k,  e_k+P^\top \nabla f_k \rangle \\
& + \frac{\Lambda \alpha_k^2}{2}\| e_k\|^2 - \frac{\Lambda\alpha_k^2}{2}\| P^\top\nabla f_k\|^2 \\
& = \alpha_k \left( \Lambda\alpha_k-1 \right) \langle P^\top\nabla f_k,  e_k \rangle  + \frac{\Lambda \alpha_k^2}{2}\| e_k\|^2+ \alpha_k \left( \frac{\Lambda\alpha_k}{2} -1 \right) \| P^\top\nabla f_k\|^2.
\end{split}
\end{equation*}
By Young's inequality with parameter $\tau_k$ and the estimate of $\norm{e_k}{}$ from Lemma \ref{lem}, 
\begin{equation*}
\begin{split}
f_{k+1}-f_k & \leq \frac{\alpha_k}{2\tau_k} \abs{\Lambda\alpha_k-1} \| P^\top\nabla f_k \|^2
+\frac{\alpha_k\tau_k}{2} \abs{\Lambda\alpha_k-1} \| e_k\|^2+ \frac{\Lambda\alpha_k^2}{2}\| e_k\|^2\\
&+ \alpha_k \left( \frac{\Lambda\alpha_k}{2} -1 \right) \| P^\top\nabla f_k\|^2 \\
& \leq -\alpha_k \left(1-\frac{\Lambda\alpha_k}{2}-\frac{1}{2\tau_k} \abs{\Lambda\alpha_k-1}\right) \| P^\top\nabla f_k \|^2\\
&+\frac{\alpha_k}{2}\left(\tau_k\abs{\Lambda\alpha_k-1} +\Lambda\alpha_k\right) \frac{\lambda^2 h_k^2 d^2}{4\ell}.
\end{split}
\end{equation*}
If $\alpha_k \leq 1/\Lambda$, choosing $\tau_k=1$, we obtain
\begin{equation}\label{uno}
\begin{split}
\alpha_k \left(1-\frac{\Lambda\alpha_k}{2}- \frac{1}{2\tau_k} \abs{\Lambda\alpha_k-1} \right) %& = \alpha_k\left(1-\frac{\Lambda\alpha_k}{2}-\frac{1}{2\omega_k}+\frac{\Lambda\alpha_k}{2\omega_k} \right)\\
%& = \frac{\alpha_k}{2 \omega_k \ell} \left[2\omega_k\ell-\alpha_k\omega_k\lambda d - \ell+\alpha_k\lambda d\right]\\
%& = \frac{\alpha_k}{2 \omega_k \ell} \left[\left(2\omega_k-1\right)\ell+\left(1-\omega_k\right)\alpha_k\lambda d \right]\\
& =\frac{\alpha_k}{2}\geq \frac{\w\alpha_k}{2}
\end{split}
\end{equation}
and
\begin{equation}\label{due}
\begin{split}
\left(\omega_k\abs{\Lambda\alpha_k-1} + \Lambda\alpha_k\right) \frac{\alpha_k \lambda^2 h_k^2 d^2}{8\ell}%& = \left(\omega_k-\omega_k \Lambda\alpha_k+\Lambda\alpha_k\right)\frac{\ell \Lambda^2 \alpha_k h_k^2}{8} \\
%& = \left(\omega_k+\left(1-\omega_k\right)\Lambda\alpha_k\right)\frac{\ell \Lambda^2 \alpha_k h_k^2}{8} \\
& = \frac{\ell \Lambda^2}{8} \alpha_k h_k^2 \leq C \alpha_k h_k^2.
\end{split}
\end{equation}
Similarly, if $1/\Lambda < \alpha_k \leq \aup < 2/\Lambda$, choose $\tau_k=\frac{\Lambda\alpha_k-1}{2-\Lambda\alpha_k-\w}$. Notice that $\tau_k>0$ because $\Lambda \alpha_k-1 >0$ and
$$0<\w<2-\Lambda\aup\leq2-\Lambda\alpha_k.$$
Then,
\begin{equation*}
\begin{split}
\alpha_k \left(1-\frac{\Lambda\alpha_k}{2}- \frac{1}{2\tau_k} \abs{\Lambda \alpha_k-1} \right) % & =\alpha_k \left(1-\frac{\Lambda\alpha_k}{2}- \frac{\Lambda \alpha_k}{2\omega_k} +\frac{1}{2\omega_k} \right)  \\
% & =\frac{\alpha_k}{2\omega_k} \left(2\omega_k-\Lambda\omega_k\alpha_k-\Lambda \alpha_k+1 \right)  \\
% & =\frac{\alpha_k}{2\omega_k} \left(1+2\omega_k-\left(1+\omega_k\right)\Lambda\alpha_k\right)  \\
& =\frac{\w\alpha_k}{2},
\end{split}
\end{equation*}
and, since $1-\w\Lambda\alpha_k\leq 1$,
\begin{equation*}
\begin{split}
\left(\tau_k\abs{\Lambda \alpha_k-1} + \Lambda \alpha_k\right) \frac{\alpha_k\lambda^2 h_k^2 d^2}{8\ell}
& = \left(\frac{1-\w\Lambda\alpha_k}{2-\Lambda\alpha_k-\w}\right) \frac{\ell \Lambda^2}{8} \alpha_k h_k^2 \leq C \alpha_k h_k^2.
\end{split}
\end{equation*}	

\end{proof}
\begin{remark}\label{rem:simple}
	If the stepsize satisfies $\alpha_k\leq 1/\Lambda$ for every $k\in\N$ (instead of the weaker condition $\alpha_k < 2/\Lambda$ from Assumption \ref{A_stepwith2}), it is clear from the proof that the result in Proposition \ref{prop:as} reduces to
	\begin{equation*}
	\begin{split}
	f_{k+1}-f_k & \leq  -\frac{\alpha_k}{2} \| P_k^\top\nabla f_k \|^2  + \frac{\ell \Lambda^2}{8} \alpha_k h_k^2  \ \ \ \ \ \ \ a.s.
	\end{split}
	\end{equation*}
	An analogous result holds for Corollary \ref{cor:as}.
\end{remark} 
% \begin{remark}\label{remk:notclassical}
%   It would be convenient to bound the term in \eqref{due} as
%     \begin{equation}\label{betterbound}
% \begin{split}
% \left(\tau_k-\Lambda\tau_k \alpha_k+\Lambda\alpha_k\right)\frac{\alpha_k \lambda^2 h_k^2 d^2}{8\ell}\leq C \alpha_k^{1+\varepsilon}
% \end{split}
% \end{equation}
% for some constant $C>0$ and some $\varepsilon>0$. The proofs in Section \ref{Sect: PL} that \eqref{betterbound} would guarantee straightforward analysis with classical techniques (and better rates) with $\alpha_k \to 0$; however it requires $\tau_k$ to vanish which is incompatible with \eqref{uno}, necessitating the more complicated analysis.
% \end{remark}

\section{Convex case}\label{Sect: Convex case}

In this section we study the  case of convex objective functions.
First, in Lemma \ref{lem:Fejer} we derive a key preliminary result 
establishing a stochastic Fejér monotonicity property for our method.
Using this Lemma we derive  convergence results  for different parameter choices. In Section~\ref{sebsec:inexact} we consider the case of a non-vanishing discretization. Then,  in Section~\ref{sec:basic}, and~\ref{sec:improv} we consider two different regimes where the discretization is made finer as the iteration proceeds, considering both a more  and a less aggressive strategy. Finally, in  Section~\ref{sec:david1} we consider the limit case where no discretization is considered, namely recursion~\ref{eqn: gradient-free}. 
Throughout this section we use the filtration $\mathcal{F}_k = \sigma(P_1, \ldots, P_{k-1})$.

\subsection{Fejér monotonicity}
We derive an a.s.  energy estimate for the Lyapunov sequence $\seq{\norm{x_k-x_*}^2}$, where $\seq{x_k}$ is the random sequence generated by Algorithm~\ref{eqn: derivative-free} and $x_*$ is any minimizer of  $f$.  
% Then, we derive several

% The inequality from the previous lemma is used to prove Theorems \ref{th:FejerWeak} and \ref{th:FejerStrong}. Under various assumptions on the sequences of the stepsizes $\seq{\alpha_k}$ and of the discretization parameters $\seq{h_k}$, we get a.s. convergence of the function values with rates in expectation and a.s. convergence of the iterates (see Theorem \ref{th:FejerWeak} and Theorem \ref{th:FejerStrong}). As expected, the results improve with stronger conditions, resulting in the differences between the two theorems.

% In the exact case of Algorithm~\ref{eqn: gradient-free} (corresponding to $h_k=0$), Lemma \ref{lem:Fejexact} is a refined version of the inequality in Lemma \ref{lem:Fejer}. It allows the convergence results in Theorem \ref{th:FejerSStrong}, where we get an asymptotic a.s. rate. 
\begin{lemma}\label{lem:Fejer}
	Assume \ref{H_lip}, \ref{H_existence} and \ref{A_Pas}, \ref{A_PE}. Let $\seq{x_k}$ be a random sequence generated by Algorithm~\ref{eqn: derivative-free}. Then, for every $k\in\N$ and $x_*\in\argmin f$,
	\begin{equation*}
\begin{split}
\mathbb{E}\left[\norm{x_{k+1}-x_*}{}^2 \ \big{|}  \mathcal{F}_k\right] - \left(1+\frac{\alpha_kh_k d^2}{2\ell\sqrt{\ell}}\right) \norm{x_{k}-x_*}{}^2 &\\ 
\leqas \quad D_k+\frac{2d}{\ell}\alpha_k^2 \norm{\nabla f_k}{}^2 + 2\alpha_k \langle \nabla f_k, x_*-x_{k} \rangle,
\end{split}
\end{equation*}
	where $\Lambda:=\lambda d/\ell$ and 
	\begin{equation}\label{def:d}
	\begin{split}
	D_k: & = \frac{\Lambda^2}{2}\left[ d\left(\alpha_k h_k\right)^2+ \sqrt{\ell} \left(\alpha_k h_k\right)\right].
	\end{split}
	\end{equation}
\end{lemma}
Lemma \ref{lem:Fejer} is the fundamental inequality used in the remainder of this section. We combine it with another important ingredient, that is the cocoercivity of $\nabla f$ (see the Baillon-Haddad Theorem, Lemma \ref{BaillHadd}) which allows us to show that the sequence $\seq{\norm{x_k-x_*}^2}$ is stochastically Fejér monotone up to the error generated at each iteration by $D_k$. Under different assumptions on the stepsize and discretization sequences the sequence $(x_k)$ exhibits different behaviors in terms of optimization and convergence rates.
\begin{proof}
	Recall the notation $\nabla_k f_k:= \nabla_{\left(P_k,h_k\right)} f(x_k)$. \text{Lemma} \ \ref{lem} provides an estimate for $\norm{e_k}$, where $e_k:=\nabla_k f_k- P_k \nabla f(x_k)$. By Assumption \ref{A_Pas}, for every $v\in\R^{\ell}$,
	\begin{equation}\label{conseq:A_Pas}
	\norm{P_k v}^2\overset{\textrm{a.s.}}{=}\frac{d}{\ell}\norm{v}^2;
	\end{equation}
	and, by \ref{A_PE}, for every fixed $w\in\R^{d}$,
	\begin{equation}\label{conseq:A_PE}
	\mathbb{E}\left[\norm{P_k^\top w}^2 \ \big| \mathcal{F}_k\right]=\norm{w}^2.\end{equation}
	Thus, for every $x_*\in\argmin f$ and every $k\in\N$, the following equality holds a.s.
		\begin{equation*}
	\begin{split}
	& \norm{x_{k+1}-x_*}{}^2 -\norm{x_{k}-x_*}{}^2\\ \quad = \quad & \norm{x_{k+1}-x_k}{}^2 + 2 \langle x_{k+1}-x_k, x_{k}-x_* \rangle\\
	(\text{Alg.~} \ref{eqn: derivative-free}) \quad = \quad  & \alpha_k^2 \norm{P_k \nabla_k f_k}{}^2  - 2\alpha_k \langle P_k \nabla_k f_k, x_{k}-x_* \rangle\\
	= \quad & \alpha_k^2 \norm{P_k \nabla_k f_k-P_k P_k^\top \nabla f_k+P_k P_k^\top \nabla f_k}{}^2 + 2\alpha_k \langle \nabla f_k, x_*-x_{k} \rangle \\
	+ & 2\alpha_k \langle P_k P_k^\top \nabla f_k -\nabla f_k, x_*-x_{k} \rangle \\
	+ & 2\alpha_k \langle P_k \nabla_k f_k- P_k P_k^\top \nabla f_k , x_*-x_{k} \rangle.
	\end{split}
	\end{equation*}
Using the fact that $(x+y)^2\leq 2x^2 + 2y^2$ and the Cauchy-Schwarz inequality, we get
\begin{equation*}
	 \begin{split}
	 & \norm{x_{k+1}-x_*}{}^2 -\norm{x_{k}-x_*}{}^2\\
	 \leq \quad  & 2\alpha_k^2 \norm{P_k e_k}{}^2 + 2\alpha_k^2 \norm{P_k P_k^\top \nabla f_k}{}^2 + 2\alpha_k \langle \nabla f_k, x_*-x_{k} \rangle \\
	+ & 2\alpha_k \langle P_k P_k^\top \nabla f_k -\nabla f_k, x_*-x_{k} \rangle + 2\alpha_k \norm{P_k e_k}{} \norm{x_*-x_{k}}{}.
	\end{split}
	\end{equation*}
	Then, by \ref{A_Pas} and Lemma \ref{lem}, 
	\begin{equation}\label{ineqforerror}
	\begin{split}    
    & \norm{x_{k+1}-x_*}{}^2 -\norm{x_{k}-x_*}{}^2\\
    \leq \quad & \frac{2\alpha_k^2d\left(\lambda h_k d\right)^2}{4\ell^2} + \frac{2\alpha_k^2d}{\ell} \norm{P_k^\top \nabla f_k}{}^2 + 2\alpha_k \langle \nabla f_k, x_*-x_{k} \rangle \\
	+ & 2\alpha_k \langle P_k P_k^\top \nabla f_k -\nabla f_k, x_*-x_{k} \rangle + \frac{\alpha_k \lambda h_k d^2}{\ell\sqrt{\ell}}  \norm{x_*-x_{k}}{}.
	\end{split}
	\end{equation}
	Finally,
	\begin{equation*}
	\begin{split}
	& \norm{x_{k+1}-x_*}{}^2 -\norm{x_{k}-x_*}{}^2 \\ \leq \quad & \frac{d^3 \lambda^2  }{2\ell^2}\left(\alpha_k h_k\right)^2 + \frac{2d}{\ell}\alpha_k^2 \norm{P_k^\top \nabla f_k}{}^2 + 2\alpha_k \langle \nabla f_k, x_*-x_{k} \rangle \\
	+ & 2\alpha_k \langle P_k P_k^\top \nabla f_k -\nabla f_k, x_*-x_{k} \rangle + \frac{ \lambda^2 d^2\left(\alpha_k h_k\right)^2}{2 \ell\sqrt{\ell}\alpha_k h_k} + \frac{\alpha_k h_k d^2}{2\ell\sqrt{\ell}} \norm{x_*-x_{k}}{}^2,
	\end{split}
	\end{equation*}
	where we used Young inequality with parameter $(\alpha_k h_kd^2)/(\ell\sqrt{\ell})$.
	Taking the conditional expectation given $\mathcal{F}_k$ and using Assumption \ref{A_PE}, we get the claim. For integrability considerations, see  Remark \ref{lemma: integrabilityI}. 
\end{proof}

Given the above inequality, we next derive a number of different convergence results considering different choices of the discretization parameter and the stepsize.

\subsection{Non-vanishing dicretization
%Inexact case
}\label{sebsec:inexact}

Using an intermediate estimate from the proof of Lemma \ref{lem:Fejer}, we get the following result: with only a boundedness assumption on the discretization sequences, we derive an upper bound on the expectation of the function values for the best iterate.
\begin{theorem}\label{th:convex_inexact}
	Assume \ref{H_lip}, \ref{H_existence}, \ref{A_Pas}, \ref{A_PE} and \ref{A_stepwith1}. Let $\seq{x_k}$ be a random sequence generated by Algorithm~\ref{eqn: derivative-free}. Then, for every $k\in\N$,
	\begin{equation*}
	\begin{split}
	\min_{j\in[k]}\mathbb{E}(f_j-f_*)&\leq \frac{\tilde{D}_k}{\sum_{j=0}^{k}\alpha_j}, 
	\end{split}
	\end{equation*}
	where
	\begin{equation*}
	\begin{split}
	    	\tilde{D}_k:&=\frac{d+\xi \ell}{2\xi\ell}\left[S_k +  \sum_{j=0}^{k}  \rho_j \left(\sqrt{S_{j-1}}+\sum_{i=0}^{j} \rho_i\right) \right],
	\end{split}
	\end{equation*}
	with $\xi:=\frac{1}{\lambda}-\frac{d\aup}{\ell}>0$, $\rho_i:=\frac{ \lambda d^2}{\ell\sqrt{\ell}} \left(\alpha_i h_i\right) $ and $S_j:=\norm{x_0-x_*}{}^2+\sum_{i=0}^{j}\frac{\Lambda^2}{2} \left(\alpha_i h_i\right)^2$. In particular, for $\alpha_k=\alpha$ and $h_k\leq \hup$, we get
	\begin{equation*}
	\begin{split}
% 	\min_{j\in[k]}\mathbb{E}(f_j-f_*)&\leq C_1/(k+1)+C_2(\hup+\hup^2)+C_3\hup^2k,
\min_{j\in[k]}\mathbb{E}(f_j-f_*)&\leq \max\{C_1/(k+1)+C_2\hup^2+C_3\hup, \ C_2 \hup^2+C_4\hup\},
	\end{split}
	\end{equation*}
	where we made explicit only the dependence on the iteration number $k$ and the bound on the discretization error $\hup$, while $C_1, C_2$ and $C_3$ are appropriate constants derived from the proof.
\end{theorem}
The above result is an extended version of Theorem~\ref{Th:main1}$(i)$.  As mentioned after Theorem~\ref{Th:main1}, it suggests that the accuracy will stop improving after a given number of iterations depending on the discretization level. The proof follows. 
% {\color{red}
%EARLY-STOPPING (OR CONVERGENCE UP TO ERROR-DOMINATED REGION)? Not seen in the practise...
%}

\begin{proof}
Start from the inequality in \ref{ineqforerror}. Recalling that we defined $\rho_j:=\frac{ \lambda d^2}{\ell\sqrt{\ell}} \left(\alpha_j h_j\right) $, for every $j\in\N$ we have that
	\begin{equation*}
	\begin{split}
	\norm{x_{j+1}-x_*}{}^2 -\norm{x_{j}-x_*}{}^2 &  \leq \frac{\Lambda^2}{2} \left(\alpha_j h_j\right)^2 + \frac{2d}{\ell} \alpha_j^2\norm{P_j^\top \nabla f_j}{}^2 + 2\alpha_j \langle \nabla f_j, x_*-x_{j} \rangle \\
	& \ \ + 2\alpha_j \langle P_j P_j^\top \nabla f_j -\nabla f_j, x_*-x_{j} \rangle+ \rho_j \norm{x_*-x_{j}}{}\\
		\left(\text{Baillon-Haddad Theorem \ref{BaillHadd}}\right) \ &\leq \frac{\Lambda^2}{2} \left(\alpha_j h_j\right)^2 + \frac{2d}{\ell} \alpha_j\norm{P_j^\top \nabla f_j}{}^2 - \frac{2}{\lambda} \alpha_j \norm{\nabla f_j}{}^2 \\
	& \ \ + 2\alpha_j \langle P_j P_j^\top \nabla f_j -\nabla f_j, x_*-x_{j} \rangle+ \rho_j \norm{x_*-x_{j}}{}.
	\end{split}
	\end{equation*}
	Taking expectations in the previous bound and denoting $u_j:=\sqrt{\mathbb{E}\norm{x_j-x_*}^2}$, we get that for every $j\in\N$
		\begin{equation*}
	\begin{split}
	u_{j+1}^2 -u_j^2 &\leq \frac{\Lambda^2}{2} \left(\alpha_j h_j\right)^2 + \frac{2d}{\ell} \alpha_j^2\mathbb{E}\left[\norm{\nabla f_j}{}^2\right] - \frac{2}{\lambda}\alpha_j \mathbb{E}\left[\norm{\nabla f_j}{}^2\right]  + \rho_j\mathbb{E}\sqrt{\norm{x_*-x_{j}}{}^2}\\
	&\leq \frac{\Lambda^2}{2} \left(\alpha_j h_j\right)^2 - 2\xi\alpha_j \mathbb{E}\left[\norm{\nabla f_j}{}^2\right] + \rho_j u_j,
	\end{split}
	\end{equation*}
	where we recall that $\xi:=\frac{1}{\lambda}-\frac{d\aup}{\ell}>0$. Then, summing from $j=0$ to $j=k$,
		\begin{equation*}
	\begin{split}
	u_{k+1}^2 +2\xi\sum_{j=0}^{k}\alpha_j \mathbb{E}\left[\norm{\nabla f_j}{}^2\right]\leq \underbrace{u_0^2+\sum_{j=0}^{k}\frac{\Lambda^2}{2} \left(\alpha_j h_j\right)^2}_{=S_{k}} +\sum_{j=0}^{k} \rho_j u_j,
	\end{split}
	\end{equation*}
    and
    \begin{equation}\label{bound1}
	\begin{split}
    \sum_{j=0}^{k}\alpha_j \mathbb{E}\left[\norm{\nabla f_j}{}^2\right]\leq \frac{1}{2\xi}\left[S_{k} +\sum_{j=0}^{k} \rho_j u_j\right]
	\end{split}
	\end{equation}
	and, by discrete Bihari's Lemma \ref{Bihari}, for every $j\in\N$
	\begin{equation}\label{bound2}
	u_j\leq \frac{1}{2}\sum_{i=0}^{j}\rho_i+\left[S_{j-1}+\left(\frac{1}{2}\sum_{i=0}^{j}\rho_i\right)^2\right]^{1/2}\leq \sqrt{S_{j-1}}+ \sum_{i=0}^{j}\rho_i.    
	\end{equation}
	Starting again from the intermediate inequality in \ref{ineqforerror}, for every $j\in\N$,
		\begin{equation*}
	\begin{split}
	\norm{x_{j+1}-x_*}{}^2 -\norm{x_{j}-x_*}{}^2 &  \leq \frac{\Lambda^2}{2} \left(\alpha_j h_j\right)^2 + \frac{2d}{\ell} \alpha_j^2\norm{P_j^\top \nabla f_j}{}^2 + 2\alpha_j \langle \nabla f_j, x_*-x_{j} \rangle \\
	& \ \ + 2\alpha_j \langle P_j P_j^\top \nabla f_j -\nabla f_j, x_*-x_{j} \rangle+ \rho_j \norm{x_*-x_{j}}{}\\
		\left(\text{convexity}\right) \ &\leq \frac{\Lambda^2}{2} \left(\alpha_j h_j\right)^2 + \frac{2d}{\ell} \alpha_j\norm{P_j^\top \nabla f_j}{}^2 - 2\alpha_j (f_j-f_*) \\
	& \ \ + 2\alpha_j \langle P_j P_j^\top \nabla f_j -\nabla f_j, x_*-x_{j} \rangle+ \rho_j \norm{x_*-x_{j}}{}.
	\end{split}
	\end{equation*}
	Taking expectations, we get
		\begin{equation*}
	\begin{split}
	u_{j+1}^2 -u_j^2 +2\alpha_j \mathbb{E}(f_j-f_*)
	&\leq \frac{\Lambda^2}{2} \left(\alpha_j h_j\right)^2 + \frac{2d}{\ell} \alpha_j\mathbb{E}\left[\norm{ \nabla f_j}{}^2\right] + \rho_j u_j.
	\end{split}
	\end{equation*}
	Summing from $j=0$ to $j=k$, we conclude the first claim:
		\begin{equation*}
	\begin{split}
	u_{k+1}^2 +2 \sum_{j=0}^{k}\alpha_j\mathbb{E}(f_j-f_*)&\leq S_{k} + \frac{2d}{\ell}\sum_{j=0}^{k} \alpha_j\mathbb{E}\left[\norm{ \nabla f_j}{}^2\right] + \sum_{j=0}^{k}  \rho_j u_j\\
	\eqref{bound1} \quad &\leq S_{k} + \frac{d}{\xi\ell}\left[S_{k} +\sum_{j=0}^{k} \rho_j u_j\right] + \sum_{j=0}^{k}  \rho_j u_j\\
	&= \frac{d+\xi \ell}{\xi\ell}\left[S_{k} +  \sum_{j=0}^{k}  \rho_j u_j\right]\\
	\eqref{bound2} \quad &\leq \frac{d+\xi \ell}{\xi\ell}\left[S_{k} +  \sum_{j=0}^{k}  \rho_j \left(\sqrt{S_{j-1}}+\sum_{i=0}^j \rho_i\right) \right]\\
	&= 2\tilde{D}_{k}.
	\end{split}
	\end{equation*}
For the second claim, for $\alpha_k=\alpha$ and $h_k\leq \hup$, define $\overline{\rho}:=\frac{ \lambda d^2}{\ell\sqrt{\ell}} \left(\alpha \hup\right)$ and note that
\begin{equation*}
	\begin{split}
	    	\tilde{D}_k & \leq \frac{d+\xi \ell}{2\xi\ell}\left[\norm{x_0-x_*}{}^2+\sum_{j=0}^{k}\frac{\Lambda^2}{2} \left(\alpha_j h_j\right)^2 +  \norm{x_0-x_*}{}\sum_{j=0}^{k}  \rho_j\right.\\
	       & \quad \quad \left. +\sum_{j=0}^{k}\rho_j\left(\sum_{i=0}^{j}\frac{\Lambda}{\sqrt{2}} \left(\alpha_i h_i\right)\right)  +\sum_{j=0}^{k}  \rho_j\left(\sum_{i=0}^{j} \rho_i \right)\right]\\
	    	& \leq \frac{d+\xi \ell}{2\xi\ell}\left[\norm{x_0-x_*}{}^2+\frac{\Lambda^2}{2} \alpha^2 \hup^2\sum_{j=0}^{k}1 + \overline{\rho} \norm{x_0-x_*}{}\sum_{j=0}^{k} 1\right. \\
	    	& \quad \quad \left. +\frac{\Lambda}{\sqrt{2}} \alpha \hup\overline{\rho} \sum_{j=0}^{k}\sum_{i=0}^{j} 1+\overline{\rho}^2\sum_{j=0}^{k} \sum_{i=0}^{j} 1 \right]\\
	    	& \leq \tilde{C}_1+\tilde{C}_2\hup^2(k+1)+\tilde{C}_3\hup(k+1) +\tilde{C}_4\hup^2(k+1)(k+2),
	    	% & \leq C_1(1+\hup+\hup^2)+C_2(\hup+\hup^2)k+C_3\hup^2k^2,
	\end{split}
	\end{equation*}
	where $\tilde{C}_1, ..., \tilde{C}_4$ are appropriate constants. Recalling that $\alpha_k$ is assumed to be constant and so that $\sum_{j=0}^k\alpha_j=\alpha (k+1)$, by trivial manipulations we get that
		\begin{equation*}
	\begin{split}
	\min_{j\in[k]}\mathbb{E}(f_j-f_*)&\leq \frac{\bar{C}_1}{k+1}+\bar{C}_2 (h+\hup^2)+\bar{C}_3 \hup^2 k, 
	\end{split}
	\end{equation*}
	where $\bar{C}_1, \bar{C}_2$ and $\bar{C}_3$ are again appropriate constants. We conclude simply by noticing that, for $k\leq 1/\hup$, the right hand side is bounded by $\bar{C}_1/(k+1)+\bar{C}_2 (h+\hup^2)+\bar{C}_3 \hup$; while, for $k>1/\hup$, the sequence $\min_{j\in[k]}\mathbb{E}(f_j-f_*)$ is non-increasing in $k$ and so controlled by the bound at $k=1/\hup$, that is $\bar{C}_1/(1/\hup+1)+\bar{C}_2 (h+\hup^2)+\bar{C}_3 \hup\leq (\bar{C}_1+\bar{C}_2+\bar{C}_3)\hup+\bar{C}_2 \hup^2$.
\end{proof}

\begin{remark}\label{rmk:Nest_Comp0}
The case of a single random direction at each iterations $(\ell=1)$ sampled from a normal distribution is studied in \cite{nesterov2011random,nesterov2017random}. Under Hypothesis \ref{H_lip} and \ref{H_existence}, \cite[Theorem 8]{nesterov2017random} states the following result:
taking constant $\alpha_k=\frac{1}{4\lambda\left(d+4\right)}$ and $h_k=h$, 
	\begin{equation}\label{eq:ourvsNest1}
	\begin{split}
	\min_{j\in [k]}\mathbb{E}\left[f(x_j)-f_{*}\right] & \leq \frac{4\lambda\left(d+4\right)\norm{x_0-x_*}{}^2}{k} + \frac{9\lambda h^2\left(d+4\right)^2}{25}.
	\end{split}
	\end{equation}
	Comparing this bound with the one obtained in Theorem~\ref{th:convex_inexact}, we see that this is  tighter. The main difference between the two approaches 
	is due to the sampling  of the random direction. The one adopted in  \cite{nesterov2017random}, is such that the expectation of the finite difference approximation of the directional derivative is the gradient of a (Gaussian) smoothing of $f$, while such a property does not hold under our assumptions, and a different proof is needed.  
	We will see in Remark~\ref{rmk:Nest_Comp}, that the different bounds lead to very similar results in terms of accuracy if the discretization error in our method is allowed to go to zero.

\end{remark}

Next, we develop our analysis for the case of a decreasing sequence of discretization parameters $(h_k)$  which allows for a finer  discretization and an increasingly  accurate approximation of the exact gradient. 

\subsection{Basic results with coarser discretization}\label{sec:basic}

We begin considering very mild assumptions on the speed at which the discretization sequence vanishes. Using Lemma \ref{lem:Fejer}  we prove a.s. convergence for the function value of the best iterate, as well as a sublinear rate in expectation. The following result is an extended version of 
 item $(ii)$ in Theorem~\ref{Th:main1}.  

\begin{theorem}\label{th:FejerWeak}
Assume  \ref{H_lip}, \ref{H_existence}, \ref{A_Pas}, \ref{A_PE}, \ref{A_stepwith1} and \ref{A_stepnotinell1}. Assume also that $\seq{\alpha_k h_k}  \in \ell^1$. Let $\seq{x_k}$ be a random sequence generated by Algorithm~\ref{eqn: derivative-free}. Then,
	\begin{equation*}
	\lim_k f(x_k) \eas f_*.
	\end{equation*}
	Moreover, we have the following convergence rate for the best iterate in expectation:
	\begin{equation*}
	\begin{split}
	\min_{j\in[k]}\mathbb{E}\left[f_j-f_*\right] \quad \leq \quad \frac{D}{\sum_{j=0}^{k}\alpha_j},
	\end{split}
	\end{equation*}
where the constant $D>0$ is provided in the proof.
\end{theorem}
	\begin{example}\label{ex1}
	For every $k\in\N$ let $h_k=h/k^r$ and $\alpha_k=\alpha/k^s$. Let $\alpha\in (0,1/\Lambda)$, $h>0$ and $s\geq 0$ with
 $
	r>1-s\geq 0.
$
	Then $\seq{\alpha_k}\notin \ell^1$, \ $\seq{\alpha_k h_k}\in\ell^1$, so the assumptions of Theorem \ref{th:FejerWeak} are satisfied. For example, the latter holds for  $\alpha_k=\alpha/k$ and $h_k=h/k^r$ with $r>0$ ($\alpha_k$ vanishing and $h_k$ going to zero arbitrarily slow); or for  $\alpha_k=\alpha$ and $h_k=h/k^r$ with $r>1$ ($\alpha_k$ constant and $h_k$ going to zero sufficiently fast).
\end{example}
\begin{proof}%[Proof of Theorem \ref{th:FejerWeak}]
Consider the inequality from Lemma \ref{lem:Fejer}, namely
			\begin{equation}\label{ineq2}
		\begin{split}
		 \mathbb{E}_k\norm{x_{k+1}-x_*}{}^2 - & \left(1+\frac{\alpha_kh_k d^2}{2\ell\sqrt{\ell}}\right) \norm{x_{k}-x_*}{}^2\\
		\leq \ & D_k+\frac{2d}{\ell}\alpha_k^2 \norm{\nabla f_k}{}^2 + 2\alpha_k \langle \nabla f_k, x_*-x_{k} \rangle \\
		\left(\text{cocoercivity of} \ \nabla f, \text{see Thm. \ref{BaillHadd}}\right) \ \leq & \  D_k+\frac{2d}{\ell}\alpha_k^2 \norm{\nabla f_k}{}^2 - \frac{2\alpha_k}{\lambda} \norm{\nabla f_k}{}^2\\
		\left(\ref{A_stepwith1}\right) \ \leq & \  D_k+2\alpha_k\left(\frac{d\aup}{\ell}- \frac{1}{\lambda}\right) \norm{\nabla f_k}{}^2\\
		= & \  D_k-2\xi\alpha_k\norm{\nabla f_k}{}^2,
		\end{split}
		\end{equation}
		where we defined $\xi:=\frac{1}{\lambda}-\frac{d\aup}{\ell}>0$. By the assumptions, $\seq{\alpha_k h_k}\in \ell^1$ and so $\seq{D_k}\in\ell^1$.	Using Robbins-Siegmund Lemma \ref{PC}, we know that $\seq{\norm{x_{k}-x_*}{}}$ is a.s. convergent for every $x_*\in\argmin f$ and that a.s.  $$\seq{\alpha_k\norm{\nabla f_k}{}^2}\in\ell^1.$$
		By Lemma \ref{lem:Fejer},
		\begin{equation}\label{needed2}
		\begin{split}
		& \mathbb{E}_k\left[\norm{x_{k+1}-x_*}{}^2\right] - \left(1+\frac{\alpha_kh_k d^2}{2\ell\sqrt{\ell}}\right) \norm{x_{k}-x_*}{}^2\\
		\leq & \ D_k+\frac{2d}{\ell}\alpha_k^2 \norm{\nabla f_k}{}^2 + 2\alpha_k \langle \nabla f_k, x_*-x_{k} \rangle \\
		\left(\text{convexity of $f$ and } \ref{A_stepwith1}\right) \ \leq & \ D_k+\frac{2d\aup}{\ell}\alpha_k \norm{\nabla f_k}{}^2 -2\alpha_k\left(f_k-f_*\right).
		\end{split}
		\end{equation}
		By the Robbins-Siegmund Lemma \ref{PC}, we get that a.s.  
		\begin{equation}\label{condl1}
		\seq{\alpha_k\left(f_k-f_*\right)}\in\ell^1.
		\end{equation}
		We know that $f_k\geq f_*$ and, by Assumption \ref{A_stepnotinell1}, the sequence $(\alpha_k)$ is positive and does not belong to $\ell^1$. So,
		\begin{equation}\label{liminf}
		\liminf_k f(x_k) \eas f_*.
		\end{equation}
		Recall that, from Proposition \ref{prop:as}, we have that \begin{equation*}
	    \begin{split}
	    f_{k+1}-f_k & \leq C \alpha_k h_k^2  \ \ \ \ \ \ \ a.s.
	    \end{split}
	    \end{equation*}
	    By the assumptions $(\alpha_k h_k)\in\ell^1$ and $h_k$ bounded, we know that also $(\alpha_k h_k^2)$ belongs to $\ell^1$. Then, from Lemma \ref{PC_det}, $f_k$ is a.s. convergent. This implies, joint with \ref{liminf}, that $\lim_k f(x_k) \eas f_*$.\\
		For the convergence rate, first take the total expectation in \eqref{ineq2}, 
		\begin{equation*}
		\begin{split}
		\mathbb{E} \norm{x_{j+1}-x_*}{}^2 - \left(1+\frac{\alpha_kh_k d^2}{2\ell\sqrt{\ell}}\right)\mathbb{E} \norm{x_{j}-x_*}{}^2 +  2\xi\alpha_j \mathbb{E}\norm{\nabla f_j}{}^2 &\leq D_j .
		\end{split}
		\end{equation*}
		Applying Lemma \ref{PC_det} to the deterministic sequence $\seq{\mathbb{E}\left[\norm{x_{k}-x_*}{}^2\right]}$ we get convergence for every $x_*\in\argmin f$ (and so the sequence is bounded above by some constant $C(x_*)$).
		Moreover, summing from $j=0$ to $j=k$,
		\begin{equation}\label{intermediate}
		\begin{split}
		\sum_{j=0}^k\alpha_j \mathbb{E}\norm{\nabla f_j}{}^2&\leq \frac{1}{2\xi } \sum_{j=0}^k\left( \mathbb{E} \norm{x_{j}-x_*}{}^2 -\mathbb{E}\norm{x_{j+1}-x_*}{}^2\right)\\
		& \quad + \frac{C(x_*)d^2}{4\xi\ell\sqrt{\ell}} \sum_{j=0}^k\alpha_j h_j+ \frac{1}{2\xi }\sum_{j=0}^k D_j \\
		& = \frac{1}{2\xi } \left(\mathbb{E} \norm{x_{0}-x_*}{}^2 -\mathbb{E}\norm{x_{k+1}-x_*}{}^2\right)\\
		& \quad + \frac{C(x_*)d^2}{4\xi\ell\sqrt{\ell}} \sum_{j=0}^k\alpha_j h_j+ \frac{1}{2\xi }\sum_{j=0}^k D_j \\
		& \leq \frac{1}{2\xi } \left(\norm{x_{0}-x_*}{}^2+ \frac{C(x_*)d^2}{2\ell\sqrt{\ell}} \sum_{j=0}^{+\infty}\alpha_j h_j+ \sum_{j=0}^{+\infty} D_j\right).
		\end{split}
		\end{equation}
		Taking the total expectation in inequality \eqref{needed2} and recalling that $\mathbb{E}\left[\norm{x_{k}-x_*}{}^2\right]$ is bounded above by some constant $C(x_*)$, we get that, for every $j\in \N$,
		\begin{equation*}
			\begin{split}
			\alpha_j\ \mathbb{E}\left(f_j-f_*\right) & \leq \frac{1}{2}\mathbb{E}\norm{x_{j}-x_*}{}^2-\frac{1}{2}\mathbb{E}\norm{x_{j+1}-x_*}{}^2+ \frac{C(x_*)d^2}{4\ell\sqrt{\ell}} \ \alpha_j h_j\\
			& \quad + \frac{D_j}{2}+\frac{d\aup}{\ell}\alpha_j \mathbb{E}\norm{\nabla f_j}{}^2.
			\end{split}
		\end{equation*}
		Summing from $j=0$ to $j=k$,
		\begin{equation*}
			\begin{split}
				& \sum_{j=0}^{k}\alpha_j \ \mathbb{E}\left(f_j -f_*\right)\\
		\leq \quad &
			\frac{1}{2}	\mathbb{E}\norm{x_{0}-x_*}{}^2-\frac{1}{2}\mathbb{E}\norm{x_{k+1}-x_*}{}^2 + \frac{C(x_*)d^2}{4\ell\sqrt{\ell}}\sum_{j=0}^{k}\alpha_j h_j \\
			& +\frac{1}{2} \sum_{j=0}^{k}D_j+\frac{d\aup}{\ell}\sum_{j=0}^k \alpha_j \mathbb{E}\norm{\nabla f_j}{}^2 \\
			\eqref{intermediate} \ \leq \quad & \frac{1}{2}\norm{x_{0}-x_*}{}^2+ \frac{C(x_*)d^2}{4\ell\sqrt{\ell}} \sum_{j=0}^{+\infty}\alpha_j h_j+\frac{1}{2} \sum_{j=0}^{+\infty} D_j\\
			& \quad  +\frac{d\aup}{2\xi\ell}\left( \norm{x_{0}-x_*}{}^2+ \frac{C(x_*)d^2}{2\ell\sqrt{\ell}} \sum_{j=0}^{+\infty}\alpha_j h_j+\sum_{j=0}^{+\infty} D_j\right)\\
			% & =  \frac{\xi\ell+d\aup}{2\xi\ell}\norm{x_{0}-x_*}{}^2+ \frac{\left(\xi\ell + d\aup\right) C(x_*)}{4\xi\ell} \sum_{j=0}^{+\infty}\varepsilon_j+\frac{\xi\ell+d\aup}{2\xi\ell} \sum_{j=0}^{+\infty} d_j\\
			= \quad &  \frac{\xi\ell+d\aup}{2\xi\ell}\left(\norm{x_{0}-x_*}{}^2+ \frac{C(x_*)d^2}{2\ell\sqrt{\ell}} \sum_{j=0}^{+\infty}\alpha_j h_j\right.\\
			&\left. \quad  +\frac{\Lambda^2}{2}\sum_{j=0}^{+\infty} \left[ d\left(\alpha_j h_j\right)^2+\sqrt{\ell} \left(\alpha_j h_j\right)\right]\right)\\
			\leq \quad & \frac{1}{2\left(1-\Lambda \aup\right)}\left(\norm{x_{0}-x_*}{}^2+ \frac{1}{2}\left(\frac{C(x_*)d^2}{\ell\sqrt{\ell}}+\Lambda d \hup +\Lambda^2\sqrt{\ell}\right) \sum_{j=0}^{+\infty}\alpha_j h_j\right)\\
			& = : D < + \infty.
			\end{split}
		\end{equation*}
		We obtain the bound by noticing that
		\begin{equation*}
			\begin{split}
				\min_{j\in[k]}\mathbb{E}\left[f_j-f_*\right] \sum_{j=0}^{k}\alpha_j & \leq \sum_{j=0}^{k}\alpha_j \ \mathbb{E}\left(f_j -f_*\right).
			\end{split}
		\end{equation*}
% 		A similar reasoning holds for the ergodic iterate, noticing that by convexity of $f$ and linearity of the expectation we have
% 		\begin{equation*}
% 			\begin{split}
% 				\mathbb{E}\left[f(\bar{x}_k) -f_*\right] & = 
% 				\mathbb{E}\left[f\left(\frac{1}{\sum_{j=1}^{k}\alpha_j}\sum_{j=1}^{k} \alpha_j x_j\right) -f_*\right] \\
% 				& \leq
% 				\mathbb{E}\left[\frac{1}{\sum_{j=1}^{k}\alpha_j}\sum_{j=1}^{k}\alpha_j\left(f_j -f_*\right)\right]\\
% 				& =\frac{1}{\sum_{j=1}^{k}\alpha_j}\sum_{j=1}^{k}\alpha_j\mathbb{E}\left(f_j -f_*\right).
% 			\end{split}
% 		\end{equation*}
\end{proof}

\subsection{Improved results with finer discretization}\label{sec:improv}
 We next make stronger assumptions on the sequences $\seq{\alpha_k}$ and $\seq{h_k}$ allowing us to derive an a.s. convergence result for the function values and a.s. convergence of the iterates to a solution.
 The following result is an extended version of Theorem~\ref{Th:main1} $(iii)$.  
  % We define the set of minimizers as
% $$\mathcal{S}:=\argmin_{x\in\R^d} \ f(x).$$
\begin{theorem}\label{th:FejerStrong}
	Under the same conditions as in Theorem \ref{th:FejerWeak}, but with the stronger Assumption \ref{A_stepbb} instead of \ref{A_stepnotinell1}. Namely, assume \ref{H_lip}, \ref{H_existence}, \ref{A_Pas}, \ref{A_PE} and $0<\alow\leq\alpha_k\leq \aup < 1/\Lambda$ and $\seq{h_k}\in \ell^1$. Let $\seq{x_k}$ be a  sequence generated by Algorithm~\ref{eqn: derivative-free}. Then,
	\begin{equation*}
	\lim_k f_k \eas f_*
	\end{equation*}
	and there exists a random variable $x_*$ with values in $\argmin f$ such that $	x_k \toas x_*.$
\end{theorem}
\begin{example}
	For every $k\in\N$, let $\alpha_k=\alpha$ constant in $\left(0,1/\Lambda\right)$ and $h_k=h/k^r$ with $h>0$ and $r>1$. Then the assumptions of Theorem \ref{th:FejerStrong} hold. These conditions are a special case of those in Example \ref{ex1}. In general, under Assumption \ref{A_stepbb} (required for Theorem \ref{th:FejerStrong}), the stepsize $\alpha_k$ is uniformly bounded below by a strictly positive constant and so it can not converge to zero. Then,  to get the condition $\seq{\alpha_k h_k}\in\ell^1$, $h_k$ can not converge to zero arbitrarily slowly as in Example \ref{ex1}. Indeed, for Theorem \ref{th:FejerStrong} to hold with $h_k$ of the form $h/k^r$, $h_k$ has to converge to zero strictly faster than $1/k$.
\end{example}	
\begin{remark}\label{rmk:rates}
	Under the assumptions in Theorem \ref{th:FejerStrong}, the convergence rate in Theorem \ref{th:FejerWeak} holds and reads as
	\begin{equation*}
	\begin{split}
	\min_{j\in[k]}\mathbb{E}\left[f_j-f_*\right] \quad 
	\leq  \quad \frac{D}{\alow k}.
	\end{split}
	\end{equation*}
\end{remark}
\begin{proof}%[Proof of Theorem \ref{th:FejerStrong}]
	From the proof of the previous theorem, we see that for every $x_*\in\argmin f$  the sequence $\seq{\norm{x_{k}-x_*}{}}$ is a.s. convergent and  a.s. $\seq{\alpha_k\left(f_k-f_*\right)}\in\ell^1$. From Assumption \ref{A_stepbb}, $0<\alow\leq \alpha_k$ and so the sequence $\left(f_k-f_*\right)$ is non-negative and belongs to $\ell^1$ a.s. In particular,
	\begin{equation*}
	\lim_k f_k \eas f_*.
	\end{equation*}
	More precisely, there is a $\bar{\Omega}\subseteq\Omega$ with $\mathbb{P}(\bar{\Omega})=1$ such that, for every $\omega \in \bar{\Omega}$, 
	\begin{equation}\label{limmm}
	    \lim_k f(x_k(\omega)) = f_*.
	\end{equation}
	For $\omega\in \bar{\Omega}$, let $x_{k_j}(\omega)$ be a convergent subsequence of $x_{k}(\omega)$; say $x_{k_j}(\omega) \to x_{\infty}$. Then, by continuity of the function $f$ and the limit in \eqref{limmm},
$$f(x_{\infty}) = \lim\limits_{k} f(x_k(\omega)) = f_*.$$
Then $x_{\infty}\in\argmin f$, as it is a  minimizer of  $f$. Summarizing, there is a full measure set for which every cluster point of the random sequence $\seq{x_k}$ belongs to $\argmin f$. Finally, combining the latter result with the fact that, for every $x_*\in\argmin f$, $\seq{\norm{x_{k}-x_*}{}}$ is a.s. convergent, the stochastic version of Opial's Lemma \ref{StochOpial} guarantees the existence of a random variable $x_*$ with values in $\argmin f$ such that $x_k \toas x_*$.
\end{proof}
In the next remarks we compare our rates on the objective function with results available in the literature. Recall that none of the considered papers prove the convergence of the iterates.
\begin{remark}\label{rmk:Nest_Comp} We compare our results for a vanishing discretization with the ones obtained in \cite[Theorem 8]{nesterov2017random} for a single direction sampled according to a normal distribution (see also Remark~\ref{rmk:Nest_Comp0}).
% Under Hypothesis \ref{H_lip} and \ref{H_existence}, \cite[Theorem 8]{nesterov2017random} obtain the following result for a single direction :
% taking constant $\alpha_k=\frac{1}{4\lambda\left(d+4\right)}$ and $h_k=h$, 
% 	\begin{equation}\label{eq:ourvsNest1}
% 	\begin{split}
% 	\min_{j\in [k]}\mathbb{E}\left[f(x_j)-f_{*}\right] & \leq \frac{4\lambda\left(d+4\right)\norm{x_0-x_*}{}^2}{k} + \frac{9\lambda h^2\left(d+4\right)^2}{25}.
% 	\end{split}
% 	\end{equation}
	%Under the same Hypothesis \ref{H_lip} and \ref{H_existence}, consider our proposed sampling strategy (see Assumptions \ref{A_Pas} and \ref{A_PE}).
	Choosing $\alpha_k=\alpha$ constant in $\left(0,1/\Lambda\right)$ and $h_k=h/k^r$ with $h>0$ and $r>1$, from Remark \ref{rmk:rates} we get that
	\begin{equation*}
	\min_{j\in[k]}\mathbb{E}\left[f(x_j) -f_*\right]  \leq  \ \ \frac{D}{\alpha k},
	\end{equation*}
	where
	$$D:=\frac{1}{2\left(1-\Lambda \alpha\right)}\left(\norm{x_{0}-x_*}{}^2+ \frac{\alpha}{2}\left(\frac{C(x_*)d^2}{\ell\sqrt{\ell}}+\Lambda d h +\Lambda^2\sqrt{\ell}\right) \sum_{j=0}^{+\infty} h_j\right).$$
For the special case $\alpha=\ell/(2\lambda d)$ and $\ell=1$ we derive (recalling that $\sum_{j=0}^{+\infty} 1/k^r =\zeta(r)<+\infty$ where $\zeta$ is the Riemann zeta function), 
	\begin{equation}\label{eq:ourvsNest2}
	\min_{j\in[k]}\mathbb{E}\left[f(x_j) -f_*\right]  \leq  \frac{2\lambda d \norm{x_{0}-x_*}{}^2}{ k} + \frac{h d^2\zeta(r)}{2 k}\left(C(x_*)+ \lambda h +\lambda^2\right). \ \ 
	\end{equation}
Comparing equations \eqref{eq:ourvsNest1} and \eqref{eq:ourvsNest2} we see that the dependence on the dimension is
the same however our result converges to the optimum because we chose a decreasing discretization parameter. In addition, we
are free to choose the stepsize bigger than the one proposed in \cite{nesterov2011random,nesterov2017random} resulting in slightly better constants. %and the benefits  from the choice of $\ell$ bigger than 1 are clear from the  bound in \eqref{eq:ourvsNest2}. 
On the other hand, in \cite{nesterov2017random}, they also study the case of accelerated inertial algorithms. A similar comparison to the one above can be done also with the results in \cite{ghadimi2013stochastic}. 
\end{remark}
\begin{remark}\label{rmk:Duchi_Comp}
The minimization of a smooth function via a zeroth-order oracle is also considered in  \cite{duchi2015optimal}. The assumptions in that paper are different from ours, both in terms of properties of the objective function, as well as of the available zeroth-order oracle. Regarding the objective function,  in addition to the Lipschitz continuity of the gradient, the authors of \cite{duchi2015optimal} require more restrictive assumptions, such as boundedness of  the gradient itself on the entire feasible set, which is assumed to be compact. The zeroth-order oracle instead is more general than ours, and consists of noisy function evaluations. In their setting $G$ denotes the bound on the gradient of $f$, and $R$ is the diameter of the feasible set. With the choice
\begin{equation}
\label{eq:ourvsDuchi0}
\alpha_k=\frac{\alpha R}{2G\sqrt{d/\ell}\sqrt{k} }\qquad \textrm{and} \qquad  h_k=\frac{u G}{\lambda d^{3/2} k},
\end{equation}
for some $\alpha>0$ and $u>0$, they derive a bound of the form
\begin{equation}\label{eq:ourvsDuchi1}
	\min_{j\in[k]}\mathbb{E}\left[f(x_j) -f_*\right]  \leq  \ \ \frac{5R G\sqrt{1+d/\ell}}{\sqrt{k}} \left( 
	\max\{\alpha,\alpha^{-1}\}+\frac{\alpha {u}^2}{\sqrt{k}}+\frac{u\log(2k)}{k} \right).
	\end{equation}
If we choose $\alpha_k=\ell/(2\lambda d)$ and $h_k=\frac{h}{\lambda d^{3/2} k^r}$ with $r>1$, we get
\begin{equation}\label{eq:ourvsDuchi2}
	\min_{j\in[k]}\mathbb{E}\left[f(x_j) -f_*\right]  \leq  \frac{2\lambda d \norm{x_{0}-x_*}{}^2}{\ell k} + \frac{H}{2 k}\left(\frac{\lambda^2 C(x_*) \sqrt{d}}{\ell\sqrt{\ell}}+ \frac{  h}{\sqrt{d} \ell} +\frac{\lambda \sqrt{d}} {\ell\sqrt{\ell}}\right) \ \ 
	\end{equation}
Comparing \eqref{eq:ourvsDuchi1} and \eqref{eq:ourvsDuchi2}, we observe that we obtain a better convergence rate, due to the fact that we consider a noise-free oracle, but our analysis leads to a worse dependence on the ratio $d/\ell$. % In addition, in our bound neither the norm of the gradient, nor the diameter of the feasible set appear. 
Since the two settings are very different the significance of the comparison is somewhat limited.

\end{remark}
Finally, in the next section, we consider the case where $h=0$. i.e. Recursion ~\eqref{eqn: gradient-free}.
\subsection{Convergence results for recursion ~\eqref{eqn: gradient-free}}\label{sec:david1}
This section covers the special case of recursion ~\eqref{eqn: gradient-free}, corresponding to the limiting case of Algorithm~\ref{eqn: derivative-free} when exact directional derivatives are available.  Lemma \ref{lem:Fejexact} provides a sharper energy estimate than the one in Lemma \ref{lem:Fejer},  which in turns  leads to the improved convergence results of Theorem~\ref{th:FejerSStrong}.  The result is an extended version of Theorem~\ref{Th:main1}$(iv)$.  
\begin{lemma}\label{lem:Fejexact}
	Assume \ref{H_lip}, \ref{H_existence}, \ref{A_Pas} and \ref{A_PE}. Let $\seq{x_k}$ be a random sequence generated by recursion~\eqref{eqn: gradient-free}. Then, for every $k\in\N$ and every $x_*\in\argmin f$,
	\begin{equation*}
	\begin{split}
	\mathbb{E}\left[\norm{x_{k+1}-x_*}{}^2 \ \big{|}  \mathcal{F}_k\right] -\norm{x_{k}-x_*}{}^2 \eas \frac{\alpha_k^2 d}{\ell} \|\nabla f_k\|^2 +2\alpha_k \langle \nabla f_k, x_*-x_k \rangle .
	\end{split}
	\end{equation*}
\end{lemma}
\begin{proof}
	For every $k\in\N$ and every $x_*\in\argmin f$, we have that a.s. 
	\begin{equation*}
	\begin{split}
	\norm{x_{k+1}-x_*}{}^2 -\norm{x_{k}-x_*}{}^2 & = \norm{x_{k+1}-x_k}{}^2 + 2 \langle x_{k+1}-x_k, x_{k}-x_* \rangle\\
	\eqref{eqn: gradient-free} \ & = \alpha_k^2 \norm{P_kP_k^\top \nabla f_k}{}^2  - 2\alpha_k \langle P_kP_k^\top \nabla f_k, x_{k}-x_* \rangle\\
	\left(\ref{A_Pas}, \text{ see } \eqref{conseq:A_Pas}\right) \ & =\frac{\alpha_k^2 d}{\ell} \langle \nabla f_k, P_k P_k^\top  \nabla f_k \rangle + 2\alpha_k \langle P_kP_k^\top \nabla f_k, x_* -x_k\rangle.
	\end{split}
	\end{equation*}
	The claim follows  taking the conditional expectation given $\mathcal{F}_k$ and using  \ref{A_PE}.
\end{proof}
Using the estimate from the above lemma and with very mild assumptions on  $\alpha_k$, we get the following result ensuring convergence of the iterates, a rate in expectation for the function values and  an asymptotic a.s. convergence rate of the form $1/k$.
\begin{theorem}\label{th:FejerSStrong}
	Assume \ref{H_lip}, \ref{H_existence}, \ref{A_Pas}, \ref{A_PE}, \ref{A_stepwith2} and \ref{A_stepnotinell1}. Let $\seq{x_k}$ be a random sequence generated by Algorithm~\ref{eqn: gradient-free}. Then there is a random variable $x_*$ with values in $\argmin f$ such that
		\begin{equation}\label{convxk}
		x_k \toas x_*.
		\end{equation}
		Moreover, the sequence $(f_k)$ is a.s. non-increasing with $\lim_k f_k \eas f_*$ and the following convergence rate in expectation holds:
		\begin{equation*}
		\begin{split}
		\mathbb{E}\left[f_k -f_*\right] & \leq D_0/\sum_{j=0}^{k}\alpha_j,
		\end{split}
		\end{equation*}
			where the constant $D_0>0$ is provided in the proof.
	Finally, if \ref{A_stepbb} also holds, 
	\begin{equation}
	f_k - f_* \eas o(k^{-1}).
	\end{equation}
\end{theorem}
%{\color{red}
%\begin{remark}
%	Notice the difference with the previous results in the case $h_k>0$. This is due to the fact that, when $h_k=0$, in Proposition \ref{prop:as} $C_k=0$ and so the sequence $f_k$ is a.s non-increasing. Then, in order to get convergence of the functional values to the minimum \eqref{convfk} and convergence of the iterates to a minimizer \eqref{convxk}, it is sufficient to require the condition $\xi_k\notin\ell^1$ (and the boundedness from below of $\xi_k$ is not needed anymore). For the same reason, we obtain rates in expectation directly on the sequence and not only for the best iterate or the ergodic one.
%\end{remark}
%}
\begin{proof}
	We recall the equality from Lemma \ref{lem:Fejexact}: for every $x_*\in\argmin f$ and every $k\in \N$,
	\begin{equation}\label{repeat}
	\begin{split}
	\mathbb{E}\left[\norm{x_{k+1}-x_*}{}^2 \ \big{|}  \mathcal{F}_k\right] -\norm{x_{k}-x_*}{}^2 \eas \frac{\alpha_k^2 d}{\ell} \|\nabla f_k\|^2 +2\alpha_k \langle \nabla f_k, x_*-x_k \rangle .
	\end{split}
	\end{equation}
	By Baillon-Haddad Theorem \ref{BaillHadd}, $\nabla f$ is $1/\lambda$ co-coercive and so 
		\begin{align}\label{eq: ineq}
\nonumber
	\mathbb{E}\left[\norm{x_{k+1}-x_*}{}^2 \ \big{|}  \mathcal{F}_k\right] -\norm{x_{k}-x_*}{}^2 &\leqas \frac{\alpha_k^2 d}{\ell} \|\nabla f_k\|^2 -\frac{2\alpha_k }{\lambda}\|\nabla f_k\|^2\\
	(\ref{A_stepwith2}) \ &\leqas - \frac{2-\aup\Lambda }{\lambda} \ \alpha_k\|\nabla f_k\|^2.
	\end{align}
	Define $\xi_0:=2-\aup\Lambda/\lambda$, a strictly positive quantity. Then by Robbins-Siegmund Lemma \ref{PC}, for every $x_*\in\argmin f$ the random variable $\seq{\norm{x_{k}-x_*}{}}$ is a.s. convergent and that a.s. $\seq{\alpha_k\|\nabla f_k\|^2}\in\ell^1.$ Beginning again from the equality in Lemma \ref{lem:Fejexact}, we estimate the term $2\alpha_k \langle \nabla f_k, x_*-x_k \rangle$ using the convexity of $f$ and the gradient inequality: 
	\begin{equation*}
	f_k + \langle \nabla f_k, x_*-x_k \rangle \leq f_*.
	\end{equation*}
	Recalling that $\alpha_k\leq\aup$ by Assumption \ref{A_stepwith2}, it leads to
	\begin{equation}\label{Fejexact:ineq}
	\begin{split}
	\mathbb{E}\left[\norm{x_{k+1}-x_*}{}^2 \ \big{|}  \mathcal{F}_k\right] -\norm{x_{k}-x_*}{}^2 + 2\alpha_k \left(f_k-f_*\right)% & \leq \frac{\alpha_k^2 d}{\ell} \|\nabla f_k\|^2\\ 
	&\leqas \frac{ \aup d}{\ell} \alpha_k\|\nabla f_k\|^2 .
	\end{split}
	\end{equation}
Robbins-Siegmund Lemma \ref{PC} reveals that $\seq{\alpha_k (f_k-f_*)} \in \ell^1$ a.s. Since by assumption $\seq{\alpha_k}$ is not summable, $\liminf_k (f_k-f_*) \eas 0$.
	By Remark~\ref{rem:h0}, the sequence $(f_k)$ is a.s. non-increasing and bounded below by $f_*$. In particular, it is a.s. convergent and $ f_k \toas f_*.$
	Following the same reasoning as in the proof of Theorem \ref{th:FejerStrong}, there is a random variable $x_*$ with values in $\argmin f$ such that $x_k \overset{\mathrm{a.s.}}{\to} x_*$.
To obtain the convergence  rate first take the total expectation in inequality \eqref{eq: ineq} and sum from $j=0$ to $j=k$ to get
\begin{equation}\label{needed}
\begin{split}
\sum_{j=0}^k \alpha_j\mathbb{E}\|\nabla f_j\|^2& \leq \frac{1}{\xi_0}\sum_{j=0}^k\left(\mathbb{E}\norm{x_{j}-x_*}{}^2- \mathbb{E}\norm{x_{j+1}-x_*}{}^2\right)\\
& = \frac{1}{\xi_0}\left(\mathbb{E}\norm{x_{0}-x_*}{}^2- \mathbb{E}\norm{x_{k+1}-x_*}{}^2\right)\\
& \leq \frac{1}{\xi_0}\norm{x_{0}-x_*}{}^2.
\end{split}
\end{equation}
Summing \eqref{Fejexact:ineq} over $k$ and combining it with \eqref{needed}, an expectation yields
	\begin{equation*}
		\begin{split}
			\sum_{j=0}^{k}\alpha_j \ \mathbb{E}\left[f_j -f_*\right]	& \leq
			\frac{1}{2}\mathbb{E}\norm{x_{0}-x_*}{}^2-\frac{1}{2}\mathbb{E}\norm{x_{k+1}-x_*}{}^2+\frac{ \aup d}{2\ell} \sum_{j=0}^{k} \alpha_j\mathbb{E}\|\nabla f_j\|^2\\
			\eqref{needed} \ & \leq\frac{1}{2}\norm{x_{0}-x_*}{}^2+\frac{ \aup d}{2\ell\xi_0} \norm{x_{0}-x_*}{}^2\\
			& = \frac{ \ell\xi_0 + \aup d}{2\ell\xi_0} \norm{x_{0}-x_*}{}^2\\
			& = \frac{1}{2-\Lambda\aup} \norm{x_{0}-x_*}{}^2 =: D_0 < + \infty.
		\end{split}
	\end{equation*}
Since $\seq{\mathbb{E}f_j}$ is non-increasing, 
$\mathbb{E}[f_k-f_*]\sum_{j=0}^k \alpha_j  \leq 	\sum_{j=0}^{k}\alpha_j \ \mathbb{E}\left[f_j -f_*\right]$. 
Dividing by the sum over $\alpha_j$ yields the rate. 
Finally, assuming also  \ref{A_stepbb} and using the fact that $\seq{\alpha_k (f_k-f_*)} \in \ell^1$ a.s., we get $\seq{f_k-f_*} \in \ell^1$ a.s. Since $\seq{f_k-f_*}$ is also a.s. non-increasing, we conclude by Lemma \ref{smallorate} that $f_k-f_* \eas o(k^{-1})$.
\end{proof}
\begin{remark}
Note that the constant $D_0$ in Theorem \ref{th:FejerSStrong} indeed corresponds to $D$ in Theorem \ref{th:FejerWeak} when the discretization is set $h_k=0$ for every $k\in\N$.
\end{remark}

\begin{remark}
Under the same conditions of Remark \ref{rmk:Nest_Comp} but considering the analogue of recursion \eqref{eqn: gradient-free}, the following result is obtained in \cite[Theorem 8]{nesterov2017random}:
taking constant $\alpha_k=1/(4\lambda\left(d+4\right))$ and $h_k=h$, 
	\begin{equation}\label{eq:ourvsNest1bis}
	\begin{split}
	\mathbb{E}\left[f(\bar{x}_k)-f_{*}\right] & \leq 4\lambda\left(d+4\right)\norm{x_0-x_*}{}^2/(2k).
	\end{split}
	\end{equation}
    Under the same assumptions, consider the case of the sampling that we proposed (see Assumptions \ref{A_Pas} and \ref{A_PE}). From Theorem \ref{th:FejerSStrong}, we get that
	\begin{equation*}
	\mathbb{E}\left[f(\bar{x}_k) -f_*\right]  \ \leq \   \frac{\norm{x_{0}-x_*}{}^2}{2\left(1-\Lambda \alpha\right)\alpha k}.
	\end{equation*}
    For the case $\alpha=\ell/(2\lambda d)$ and $\ell=1$ we derive
	\begin{equation}\label{eq:ourvsNest2bis}
	\mathbb{E}\left[f(\bar{x}_k) -f_*\right]  \leq  2\lambda d \norm{x_{0}-x_*}{}^2/k 
	\end{equation}
	and the same observations of Remark \ref{rmk:Nest_Comp} hold.
\end{remark}
\section{Polyak-Łojasiewicz case}\label{Sect: PL}
% More precisely, Corollary \ref{cor:as} leads to Theorem \ref{th:PL_lin}; while Proposition \ref{prop:as} gives rise - first - to Theorem \ref{th:PL_alphaconst} by Lemma \ref{lem:aux} and - second - to Theorem \ref{th:Chung} by the application of Chung's Lemma \ref{chung}. 
In contrast with the case of a general convex $f$ considered thus far, this section assumes the PL inequality (\ref{H_PL}), but convexity (\ref{H_existence}) is not needed. Since strong convexity implies the PL inequality, all of the results in this section hold when the objective function is strongly convex. We use the PL inequality in the a.s. quasi-decreasing estimates of Section \ref{sec:as} to get the main estimate of Lemma \ref{lem:mainineq}. The application of  Lemma \ref{lem:mainineq} in different settings leads to the convergence rates in expectation for the function values given in Theorems \ref{th:PL_lin}, \ref{th:PL_alphaconst} (based on Lemma \ref{lem:aux}), \ref{th:Chung} (based on Lemma \ref{chung}) and \ref{th:linrate}.
These results are similar but intrinsically different. In Theorem \ref{th:PL_lin}, we study the case in which both sequences $\seq{h_k}$ and $\seq{\alpha_k}$ are bounded above, but not converging to zero; specifically, the error generated by the discretization does not vanish. In this context, we obtain a \emph{linear} rate in expectation not to the optimal value, but to a sublevel of the objective function depending on $\Cupb, \tlow$ and $\gamma$, see \eqref{error}. In Theorem \ref{th:PL_alphaconst}, for the case of $\alpha_k$ constant and vanishing $h_k$, we get \emph{sublinear} rates in expectation \emph{to the optimum}. In Theorem \ref{th:Chung} we obtain similar rates assuming that both $\seq{\alpha_k}$ and $\seq{h_k}$ converge to zero polynomially. Note that the algorithm does not converge to the optimal value if $h_k$ does not vanish, even with rapid decay of $\alpha_k$, a fact that may be surprising to readers more familiar with first-order stochastic approximation algorithms though, as discussed,  it is easy to see why. Finally, in Theorem \ref{th:linrate} we show linear convergence rates to the optimal value for a fast decay of $h_k$ .

\subsection{Main estimate}
The following basic estimate will be used repeatedly. 
\begin{lemma}\label{lem:mainineq}
Let $\seq{x_k}$ be generated by Algorithm~\ref{eqn: derivative-free}. Assume \ref{H_lip}, \ref{H_PL}, \ref{A_Pas}, \ref{A_PE} and \ref{A_stepwith2}. Then, for every $k\in\N$,
	\begin{equation*}
	\begin{split}
	\mathbb{E}\left[f_{k+1}-f_*\right] & \leq  \left(1-\frac{\w\alpha_k\gamma}{2}\right) \mathbb{E}\left[f_{k}-f_*\right]  + C \alpha_k h_k^2.
	\end{split}
	\end{equation*}
	Assuming also \ref{A_stepbb}  and defining $\eta:=1-\w\alow\gamma/2$, we get
	\begin{equation*}
	\begin{split}
		\mathbb{E}\left[f_{k}-f_*\right]
		& \leq \eta^k \left[\left(f_{0} -f_*\right)  + \frac{C}{\eta} \sum_{j=0}^{k-1}\frac{\alpha_j h_j^2}{\eta^j} \right].
	\end{split}
	\end{equation*}
\end{lemma}
\begin{proof}
Taking the conditional expectation given $\mathcal{F}_k$ in the a.s. inequality of Proposition \ref{prop:as}, we have
	\begin{equation*}
	\begin{split}
	\mathbb{E}\left[f_{k+1} \ \big{|}  \mathcal{F}_k\right]-f_k & \leq  -\frac{\w\alpha_k}{2} \mathbb{E}\left[\| P_k^\top\nabla f_k \|^2 \ \big{|}  \mathcal{F}_k\right]  + C \alpha_k h_k^2\\
	% & =  -\tau_k \mathbb{E}\left[\nabla f_k^\top P_k P_k^\top\nabla f_k \ \big{|}  \mathcal{F}_k\right]  + C_k\\
	%& =  -\tau_k \nabla f_k^\top \mathbb{E}\left[P_k P_k^\top \ \big{|}  \mathcal{F}_k\right]\nabla f_k  + C_k\\
	% (\ref{A_PE}) \ & =  -\tau_k\nabla f_k^\top\nabla f_k  + C_k\\
	(\ref{A_PE}) \ & =  -\frac{\w\alpha_k}{2} \norm{\nabla f_k}{}^2  + C \alpha_k h_k^2\\
	(\ref{H_PL})  \ & \leq  -\frac{\w\alpha_k}{2} \gamma \left(f_k-f_*\right)  + C \alpha_k h_k^2.
	\end{split}
	\end{equation*}
	An expectation yields the first claim. Under Assumption \ref{A_stepbb}, the last inequality yields
    \begin{equation*}
	\begin{split}
	\mathbb{E}\left[f_{k+1}-f_*\right] & \leq  \left(1-\frac{\w\alow\gamma}{2}\right) \mathbb{E}\left[f_{k}-f_*\right]  + C \alpha_k h_k^2.
	\end{split}
	\end{equation*}
	Iterating leads to the second claim:
	\begin{equation*}
	\begin{split}
		\mathbb{E}\left[f_{k}-f_*\right]
		& \leq  \eta^k\left(f_{0} -f_*\right)\\
		&+ C \left[ \alpha_{k-1}h_{k-1}^2+\eta \alpha_{k-2}h_{k-2}^2+\eta^2 \alpha_{k-3}h_{k-3}^2+...+\eta^{k-1}\alpha_0 h_{0}^2\right]\\
	& =  \eta^k \left(f_{0} -f_*\right)  + C \sum_{j=0}^{k-1}\eta^{k-1-j}\alpha_j h_j^2\\
	% & =  \left(1-\tlow \gamma\right)^k \left(f_{0} -f_*\right)  + \Cupb \frac{1-\left(1-\tlow \gamma\right)^k}{1-\left(1-\tau \gamma\right)}\\
	& =  \eta^k \left[\left(f_{0} -f_*\right)  + \frac{C}{\eta} \sum_{j=0}^{k-1}\frac{\alpha_j h_j^2}{\eta^j} \right].
	\end{split}
	\end{equation*}
\end{proof}
\subsection{Linear quasi-rate} % ($\seq{\alpha_k}$ and $\seq{h_k}$ bounded)
We first provide an extended version of Theorem~\ref{Th:main2}$(i')$. In this case both $\seq{\alpha_k}$ and $\seq{h_k}$ are bounded above but not vanishing, leading to the following result as a direct consequence of Lemma \ref{lem:mainineq}. In particular, the following bound suggests to stop iterating when $(f_k - f_*) \leq C \aup\hup^2/(1-\eta)$.
\begin{theorem}\label{th:PL_lin}
	Let $\seq{x_k}$ be generated by Algorithm~\ref{eqn: derivative-free}. Assume \ref{H_lip}, \ref{H_PL}, \ref{A_Pas}, \ref{A_PE}, \ref{A_stepwith2} and \ref{A_stepbb}. Then, for every $k\in\N$,
	\begin{equation*}
	\begin{split}
	\mathbb{E}\left[f_{k}-f_*\right] & \leq \eta^k \left(f_0-f_*\right)  + \frac{C \aup\hup^2}{1-\eta}\left[1-\eta^k\right],
	\end{split}
	\end{equation*}
	where the constant $C$ is defined in Proposition \ref{prop:as} and $\eta=1-\w\alow\gamma/2$.
\end{theorem}

% {\color{red}
% \begin{remark}
% In particular, the previous bound does not suggest to early-stop the algorithm, but only the following strategy in two different cases:\textcolor{green}{If we're in the ball, stop}
% \begin{itemize}
%     \item if $(f_0-f_*)\leq C \aup\hup^2/(1-\eta)$, stop at the iteration $k=0$;
%     \item otherwise, run the algorithm as long as possible ($k\to +\infty$).
% \end{itemize}
% \end{remark}
% }

\begin{remark}
For recursion~\ref{eqn: gradient-free} we recover the linear rate  proved in \cite{kozak2019stochastic}; namely,
\begin{equation*}
	\begin{split}
		\mathbb{E}\left[f_{k}-f_*\right] & \leq \left(1-\frac{\w\alow\gamma}{2}\right)^k \left(f_0-f_*\right).
	\end{split}
\end{equation*}
On the other hand, for  $0\leq h_k \leq \hup$, the cumulative error term does not vanish:
\begin{equation}\label{error}
    \lim_{k} \ \frac{C \aup\hup^2}{1-\eta}\left[1-\eta^k\right] =\frac{2C \aup\hup^2}{\w \alow\gamma} = \frac{\ell \Lambda^2 \aup \hup^2}{4\w\alow \gamma \min\left(1, 2-\Lambda\aup-\w\right)},
\end{equation}
where we recall that $0<\w<2-\Lambda \aup$ \  and $ \w\leq 1$. Finally, note that for  $\alpha_k=1/\Lambda$, by Remark \ref{rem:simple} we have that the decreasing rate is $1-\gamma/(2\Lambda)$.
\end{remark}

% {\color{red}
% \begin{remark} Comparison with Nesterov: assume \ref{H_lip}, \ref{H_existence} and that $f$ is strongly convex with constant $\gamma$. Take $\alpha=\frac{1}{4\lambda\left(d+4\right)} $. Then, for $D=\frac{18\lambda h^2\left(d+4\right)^2}{25\gamma}$,
% 	\begin{equation*}
% 	\begin{split}
% 	\mathbb{E}\left[f_{k}-f_*\right] & \leq \frac{\lambda}{2}\left[D+\left(1-\frac{\gamma}{8\lambda\left(d+4\right)}\right)^k\left(\norm{x_0-x_*}{}^2-D\right)\right].
% 	\end{split}
% 	\end{equation*}
% \end{remark}
% }
\subsection{Sublinear rates} % ($\seq{\alpha_k}$ bounded and $\seq{h_k}$ vanishing)
We now state the results obtained with bounded step-size and vanishing discretization. Assuming only that $\seq{h_k}$ converges to zero, the objective function values converge in expectation to the optimum; while, for a polynomial decay of $h_k$, we get sublinear convergence rates. The following is an extended version of Theorem~\ref{Th:main2}  $(ii')$.  
\begin{theorem}\label{th:PL_alphaconst}
		Let $\seq{x_k}$ be generated by Algorithm~\ref{eqn: derivative-free}. Assume \ref{H_lip}, \ref{H_PL}, \ref{A_Pas}, \ref{A_PE}, \ref{A_stepwith2} and \ref{A_stepbb}. If $\seq{h_k}$ converges to zero, then
	\begin{equation*}
	\begin{split}
	\lim_{k} \mathbb{E} f_{k}=f_*.
	\end{split}
	\end{equation*}
	Moreover, if $h_k=h/k^r$ for some $r>0$ and $h>0$, then
	\begin{equation*}
	    \limsup_{k} \  k^{2r}\mathbb{E}\left[f_{k}-f_*\right]\leq \frac{2C \aup h^2}{\w \alow \gamma},
	\end{equation*}
    In particular, there is a constant $\tilde{C}>0$ such that 
	$$\mathbb{E}\left[f_{k}-f_*\right]\leq \tilde{C}/k^{2r}$$
	and, for every $t\in\left(0,r\right)$,
	$$\mathbb{E}\left[f_{k}-f_*\right]=o\left(\frac{1}{k^{2t}}\right).$$
\end{theorem}
\begin{proof}
	 From  Lemma \ref{lem:mainineq}, we get
	\begin{equation*}
	\begin{split}
	\mathbb{E}\left[f_{k+1}-f_*\right] & \leq  \eta \mathbb{E}\left[f_{k}-f_*\right]  + C \aup \ h_{k}^2.
	\end{split}
	\end{equation*}
	The remainder follows from Lemma \ref{lem:aux} with $c_k=C \aup h^2_k$, $c=C \aup h^2$ and $t=2r$.
\end{proof}
% \subsection{Sublinear rates} % ($\alpha_k$ and $h_k$ vanishing)

In the next result we allow both the step-size and the discretization to converge to zero polynomially. In this case, we get again sublinear rates in expectation similar to the ones obtained in Theorem \ref{th:PL_alphaconst}.
\begin{theorem}\label{th:Chung}
	Let $\seq{x_k}$ be generated by Algorithm~\ref{eqn: derivative-free}. Assume \ref{H_lip}, \ref{H_PL}, \ref{A_Pas} and \ref{A_PE}. For $0<\alpha<2/\Lambda$ and $h>0$, set $\alpha_k= \alpha/k^s$, $h_k=h/k^r$ and define
	\begin{eqnarray}\label{cc}
	c  :=& \w \alpha \gamma /2 \quad \text{and} \quad d:=& C \alpha h^2.
\end{eqnarray}
Then, for $s=1$ and $r>0$, we get
	\begin{equation}
\mathbb{E}\left[f_{k}-f_*\right]\leq 
\begin{cases}
\frac{d}{\left(c-2r\right)k^{2r}}+o\left(\frac{1}{k^{2r}}\right) \ \ \ \ \ & \text{if} \ \ 2r<c; \\
 
O\left(\frac{\log k}{k^{c}} \right)\ \ \ \ \ & \text{if} \ \ 2r = c ;\\
 
O\left(\frac{1}{k^{c}}\right) \ \ \ \ \ & \text{if} \ \ 2r>c.
\end{cases}
\end{equation}
If $0<s<1$, for every $r>0$, we have
$$\mathbb{E}\left[f_{k}-f_*\right]\leq \frac{d}{c} \frac{1}{k^{2r}}+o\left(\frac{1}{k^{2r}}\right).$$
\end{theorem}
\begin{remark}
In order for the previous results to hold, both $\seq{\alpha_k}$ and $\seq{h_k}$ must converge to zero. In the case of $s=1$ (and so $\alpha_k$ proportional to $1/k$), as the intuition suggests, the rate improves for larger $r$ (and thus for $h_k$ vanishing faster) up to the value $c/2$. But eventually a saturation effect occurs: increasing $r$ beyond $c/2$ does not improve the bound. On the other hand, for $\alpha_k=\alpha/k^s$ with $0<s<1$, the convergence rates improve with larger values of $r$ similar to the ones in Theorem \ref{th:PL_alphaconst}. %{\color{red} Observe that the results of Theorem \ref{th:PL_alphaconst} and Theorem \ref{th:Chung} are equivalent and the convergence rates are not improved by letting $\alpha_k$ go to zero.}
\end{remark}
\begin{remark}
Note the difference between the results of this section. In Theorem \ref{th:PL_lin}, for non-vanishing $h_k$, we obtain a linear rate \emph{with an  error}; while in Theorems \ref{th:PL_alphaconst} and \ref{th:Chung}, with $h_k$ going to zero, the convergence rates are sublinear but \emph{to the optimum}.
\end{remark}
\begin{proof}%[Proof of Theorem \ref{th:Chung}]
	From Lemma \ref{lem:mainineq},
	\begin{equation*}
	\begin{split}
	\mathbb{E}\left[f_{k+1}-f_*\right] & \leq  \left(1-\frac{\w\alpha_k\gamma}{2}\right) \mathbb{E}\left[f_{k}-f_*\right]  + C \alpha_k h_k^2.
	\end{split}
	\end{equation*}
    To conclude the first result, apply the first part of Lemma \ref{chung} with $p=2r$ and $c,\ d$ as in \eqref{cc}. For the second result, apply the second part of Lemma \ref{chung} with the same $c,\ d$ as before and
$
	s = s \quad \text{and} \quad t = s+2r.
$
\end{proof}
\subsection{Linear rate}
Finally, assuming a fast decay of $h_k$, we derive linear convergence rates in expectation. The following result is an extended version of Theorem~\ref{Th:main2}$(iii')$. The proof is  a simple consequence of Lemma \ref{lem:mainineq}.
\begin{theorem}\label{th:linrate}
    Let $\seq{x_k}$ be generated by Algorithm~\ref{eqn: derivative-free}. Assume \ref{H_lip}, \ref{H_PL}, \ref{A_Pas}, \ref{A_PE}, \ref{A_stepwith2} and \ref{A_stepbb}. Assume that $\seq{h_k^2/\eta_k}\in \ell^1$. Then, for every $k\in\N$,
		\begin{equation*}
	\begin{split}
		\mathbb{E}\left[f_{k}-f_*\right]
		& \leq \eta^k \left[\left(f_{0} -f_*\right)  + \frac{C\aup}{\eta} \sum_{j=0}^{+\infty}\frac{h_j^2}{\eta^j} \right],
	\end{split}
	\end{equation*}
	where the constant $C$ is defined in Proposition \ref{prop:as} and $\eta=1-\w\alow\gamma/2$.

\end{theorem}
\section{Numerical results}\label{sec:num}
In this section we present synthetic examples  illustrating the different results we derived and discussed. Our analysis unifies many algorithms that have been thoroughly empirically examined, e.g., \cite{pmlr-v80-choromanski18a,berahas2019theoretical,Kozak2021Stochastic}. Hence, we present only toy problems to illustrate our theoretical results.
% are achievable in practice and refer the interested reader to the aforementioned works for more sophisticated examples and comparisons. 
 We omit the convex, non-PL case as we were unable to identify a function that resulted in a qualitative difference in performance of the algorithm for any $\ell \in \{1, \ldots, d\}$.\\
% \paragraph{Convex function}
% For the convex function we use $f(x) = \frac{1}{2}\norm{Ax}^2 \mathbbm{1}(\norm{Ax} < 1) + (\norm{Ax}-\frac{1}{2})\mathbbm{1}(\norm{Ax} \geq 1)$ with $n=d=100$ and $\lambda = 100$. This is the well-known Huber loss function for a quadratic function. For the cases $\ell < d$ we take the average of 10 runs. We provide figures showing progress towards the minimum and towards a minimizer.  
{\bf Convex function satisfying PL inequality.}
Many convex functions satisfy the PL inequality leading to an improved rate of convergence. An example of this case is $f(x) = \norm{Ax}^2$ where $A \in \reals^{n\times d}$ is fixed but not necessarily full column rank and $\lambda = 100$. Specifically, if any eigenvalue of $A$ is 0 then $f$ is not strongly convex, however because it is PL we are still able to apply Theorem \ref{Th:main2}. We choose $n=d=100$, and we force at least one eigenvalue of $A$ to be 0. For the cases $\ell < d$ we take the average of 10 runs. While in the long run the discrete gradient method catches up to the subspace approaches, it is important to recognize that for many practical problems of interest the dimension of the objective function may be very high relative to the budget for function evaluations (this budget could be due to time, money, computational power, etc). In very high-dimensional cases,  it may not even be possible to perform a single iteration of gradient descent, due to the $d+1$ function evaluations required at each iteration; requiring only $\ell$ function evaluations per iteration, may allow for substantial progress with identical budget constraints. This effect is even more apparent with the rapid initial progress made with  $\ell < d$  in the left panel of Figure \ref{fig:convexpl}.\\
\begin{figure}[t]
    \centering
    \includegraphics[width=.32\linewidth]{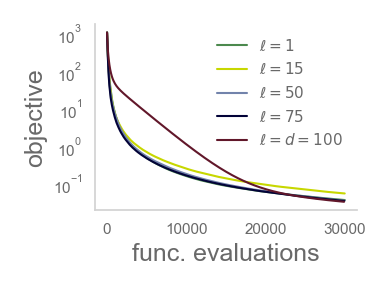}
    \includegraphics[width=.32\linewidth]{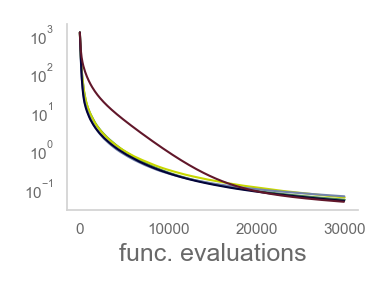}
    \includegraphics[width=.32\linewidth]{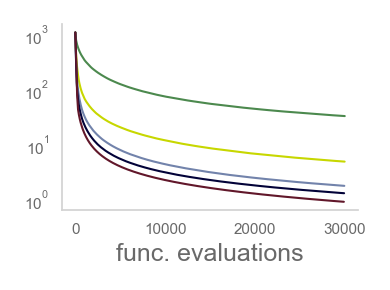}
    \caption{Optimization of a convex objective function satisfying the PL inequality. \textbf{left:} $\alpha = \ell/(d\lambda)$, $h=10^{-7}$. \textbf{center:} $\alpha = \ell/(d\lambda)$, $h=10^{-7}/k^{0.0001}$. \textbf{right:} $\alpha = \ell/(d\lambda\sqrt{k})$, $h=10^{-7}/k^{0.0001}$. Note the different axes between the figures.}
    \label{fig:convexpl}
\end{figure}
{\bf Non-convex function satisfying PL inequality.} $f(x) = \norm{Ax}^2 + 3\sin^2(c^\top x)$, with $A$ fixed but not necessarily full rank and $Ac = c$. Again,  let $n = d =100$ and $\lambda = 100$. For the cases $\ell < d$ we take the average of 10 runs.
\begin{figure}[t]
    \centering
    \includegraphics[width=.31\linewidth]{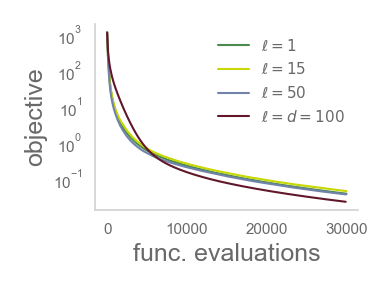}
    \includegraphics[width=.31\linewidth]{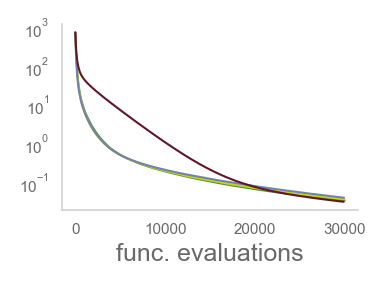}
    \includegraphics[width=.31\linewidth]{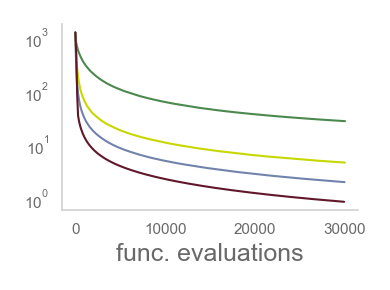}
    \caption{Optimization of a non-convex objective function satisfying the PL inequality. \textbf{left:} $\alpha = \ell/(d\lambda)$, $h=10^{-7}$. \textbf{center:} $\alpha = \ell/(d\lambda)$, $h=10^{-7}/k^{0.0001}$. \textbf{right:} $\alpha = \ell/(d\lambda\sqrt{k})$, $h=10^{-7}/k^{0.0001}$. Note the different axes between the figures.}
    \label{fig:nonconvexpl}
\end{figure}
The most notable feature  in Figures \ref{fig:convexpl} and \ref{fig:nonconvexpl} is that when the step-size is not fixed, choosing $\ell < d$ severely under performs the discrete gradient method. When $\ell < d$ only a subset of the available information is being used at each iteration, and with the step-size diminishing every successive iteration has less impact than those that precede it. Thus, the trade-off between cost-per-iteration and progress-per-iteration favors a higher per-iteration cost in return for more progress, particularly in the early iterations. This trade-off flips when the step-size is fixed: much faster progress is made early on when $\ell<d$ and many directions provide improvement of the objective, but of course the discrete gradient method ultimately catches up.\\ %It is possible  in the non-convex case we achieve better performance with the subspace methods when $h$ is decaying, while in the convex PL case the performance is better when $h$ is fixed.
{\bf Variability due to stochasticity}.
One potential benefit of letting $\ell=d$ and performing the full discrete gradient method is that there is no randomness involved so the results are deterministic. The theorems provide guarantees for $\Expectation f(x_k)-f_*$, but here we investigate
how much variability can be expected between runs with identical initializations when $\ell < d$. We use the same non-convex function as previously, and perform 100 runs using the same initialization in each case,  considering 15000 function evaluations. The substantial overlap between the $\ell<d$ cases makes it difficult to discern between various values of $\ell$, thus in this figure we provide only the extremes, $\ell = 1$ and $\ell = d$. The bold line represents the mean of the $\ell=1$ case, and since $\ell=d$ is deterministic it is run only once.
\begin{figure}[t]
    \centering
    \includegraphics[width=.31\linewidth]{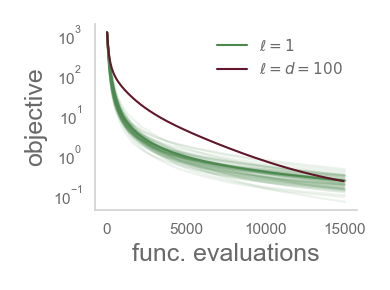}
    \includegraphics[width=.31\linewidth]{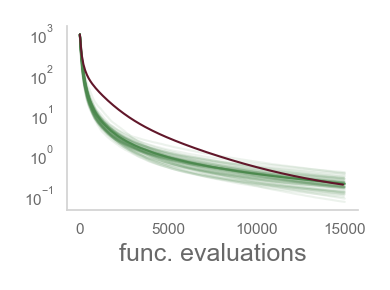}
    \includegraphics[width=.31\linewidth]{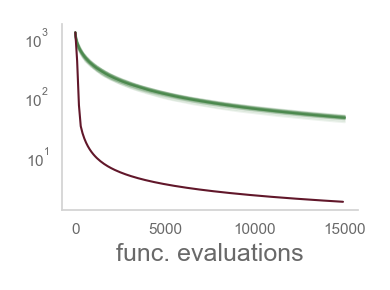}
    \caption{Optimization of a convex objective function with 100 restarts at the same initialization.  \textbf{left:} $\alpha = \ell/(d\lambda)$, $h=10^{-7}$. \textbf{center:} $\alpha = \ell/(d\lambda)$, $h=10^{-7}/k^{.0001}$. \textbf{right:} $\alpha = \ell/(d\lambda\sqrt{k})$, $h=10^{-7}/k^{1.0001}$. Note the different axes between the figures.}
    \label{fig:simulations}
\end{figure}
As expected, there is variability between runs when $\ell < d$, however in the early iterations even the worst case performs better than discretized gradient descent. We have only a heuristic explanation for this phenomenon: in the early iterations there are many directions that may lead to improvement so the inexpensive $\ell=1$ directions are more efficient than a full gradient estimate, but as we approach the optimum a judicious choice of direction is rewarded. The fact that the discrete gradient method catches up and eventually outperforms the others is consistent with the theory since % the incremental error due to finite differences accrues at every iteration and, 
all else equal, larger $\ell$ implies better progress per iteration. Of greater practical interest is the out-performance in the scenario $\ell \ll d$ in the early iterations. This suggests that for low-precision optimization, or when relatively few iterations are possible due to time or money constraints, it may be beneficial to choose $\ell <d$. Indeed, this is precisely the scenario where subspace descent methods are used, cf. \cite{Kozak2021Stochastic,pmlr-v80-choromanski18a,NIPS2018_7451}. Our theory does not cover the use of a backtracking line search such as the one presented in \cite{berahas2019global}, so we do not provide figures detailing its performance; however, in practice a line search is a necessary component to achieve outperformance compared to the gradient method as discussed at length in \cite{Kozak2021Stochastic}.\\
{\bf Effect of $h$.} Recall that there is an additive error term at each iteration due to the use of finite differences to approximate the gradient. In light of this fact, it may be surprising that the preceding figures appear to have objective function values that decrease monotonically with the increase in function evaluations even when $h$ is fixed. Note, however, that for a fixed $h$ as in Theorem \ref{th:PL_lin} the  error is asymptotically $\mathcal{O}(h^2)$, whereas in Theorem \ref{th:PL_alphaconst} a decaying $h_k = h/k^r$ is used with $r, h >0$, resulting in an objective that decays to zero.
Figure \ref{fig:h_examples} demonstrates that the limiting error, or lack thereof, guaranteed by the theorems is observed in practice. The first figure is a convex function satisfying the PL inequality with various fixed values for $h$.  The second is a non-convex function satisfying the PL inequality with the same values of $h$, and in the third figure we set $h_k = h/k^r$ with $h=10^{-5}$ and $r = 1$ and run the algorithm 100 times on the non-convex function. In all cases, $d = 5$, $\ell = 1$ and $\lambda = 4$, resulting in fast convergence. \\
Several conclusions can be drawn from Figure \ref{fig:h_examples}. First, as expected by the theorems, $h$ does not appear to play a role in the rate of convergence, only in the magnitude of the asymptotic error. % This is important because it suggests that it may be possible to rely on the  simpler analysis of recursion \eqref{eqn: gradient-free}, as in \cite{Kozak2021Stochastic} and \cite{kozak2019stochastic}. 
% The latter  provide probabilistic theoretical guarantees and an accelerated version when directional derivatives are available.  Ultimately, this will not be verified until proofs in the derivative free setting are available, but these figures suggest there is reason to be optimistic. 
Further, for all of the algorithms that fit our theory the finite difference error can essentially be ignored,  provided the desired accuracy is less than $\mathcal{O}(h^2)$. So, with $h$ the square root of machine precision, our theorems seem to indicate (and our figures support) that the error due to finite differences can be safely ignored. The figure on the right shows that the variance due to different stochastic realizations of the algorithm is small, suggesting that the theoretical analysis done in \cite{Kozak2021Stochastic} and \cite{kozak2019stochastic} may be extended to the discrete setting.

\begin{figure}[ht]
    \centering
    \includegraphics[width=.31\linewidth]{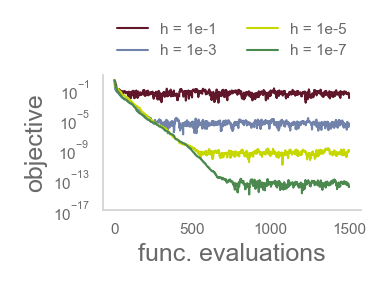}
    \includegraphics[width=.31\linewidth]{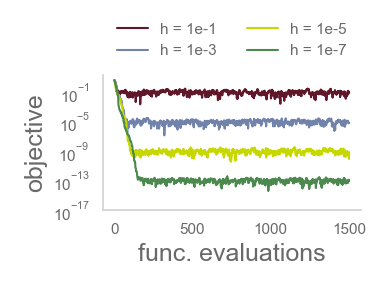}
    \includegraphics[width=.31\linewidth]{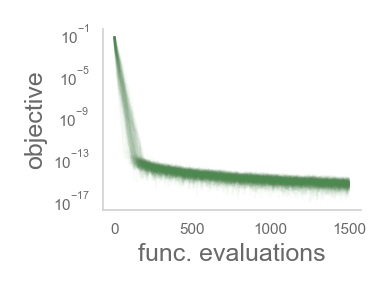}

    \caption{\textbf{left:} Optimization of a convex PL function using various values of $h$. \textbf{center:} Optimization of a non-convex PL function using various values of $h$. \textbf{right:} Optimization of a non-convex PL function with $h_k = 10^{-5}/k$ }
    \label{fig:h_examples}
\end{figure}

\section{Conclusions}\label{sec:conc}
We presented a method that generalizes several well-known derivative-free optimization algorithms including, for example, spherical smoothing and discretized versions of coordinate and gradient descent. 
We provide  convergence analysis of this generic method 
considering objective functions that are either convex or satisfy a Polyak-Łojasiewicz (PL) condition. Multiple possible choices for the stepsizes and the finite difference  parameter are studied.  The best choice depends on the error tolerance of the user. By allowing the stepsize and discretization to decay the algorithm achieves a slower rate of convergence but is able to converge to an optimum, but by fixing these values a faster convergence rate is obtained at the expense of converging only to within a region of an optimum.
To the best of our knowledge, this work provides the first convergence guarantees for the iterates of spherical smoothing and discretized coordinate descent to a minimizer when the objective function is convex. 

There are several possible extensions of this work.
Our analysis does not cover the use of an adaptive stepsize that has been shown empirically to be highly advantageous  \cite{NIPS2018_7451,Kozak2021Stochastic}. Theoretical analysis using a stochastic linesearch may be possible, several recent papers offer promising results that may extend to our case \cite{cartis2018global,berahas2019global,MR4060460}. It would be interesting to analyze 
%this framework in the setting of subsamples the observations, as for example in empirical risk minimization. Equivalently, analyzing 
the algorithm in the presence of noisy function evaluations as in \cite{MR50243,MR499560}. Such an extension would be of great practical consequence as many applications of interest have noisy objective functions. Finally, faster convergence may be possible using derivative-free quasi-Newton methods or any methods that exploit the curvature of the objective as in \cite{MR3936038,pmlr-v119-hanzely20a,pmlr-v119-hanzely20b,bollapragada2019adaptive}. 

% Non-BibTeX users please use
\bibliographystyle{siam}
\bibliography{Bibliography.bib}

\clearpage
\section{Supplementary material and auxiliary lemmas}
Here we collect the main auxiliary results used in the convergence analysis of Algorithm~\ref{eqn: derivative-free}.
 
\paragraph{Proof of Lemma \ref{lem}:}
Let $k\in\N$,	denoting by $p^{(j)}$ the $j$-th column of $P_k$, we want to estimate 
\begin{equation}\label{eq}
\norm{\nabla_{\left(P_k, h_k\right)} f(x)-P_k^\top \nabla f(x)}{} = \sqrt{\sum_{j=1}^{\ell} \left(\left[\nabla_{\left(P_k, h_k\right)} f(x)\right]_j-\langle  \nabla f(x), p^{(j)} \rangle \right)^2}.
\end{equation}
To get an upper-bound for the term in parenthesis, we use \ref{H_lip}. As $\nabla f$ is $\lambda$-Lipschitz, the Descent Lemma \ref{DescentLemm} holds: for every $x\in\R^d$,
\begin{equation}\label{mmm}
\abs{f(x+h_k p_j) - f(x) - h_k\langle \nabla f(x), p^{(j)} \rangle } \leq \frac{\lambda h_k^2}{2}\norm{p^{(j)}}^2.
\end{equation}
Then, rearranging,
\begin{equation}\label{lb}
\begin{split}
\abs{\left[\nabla_{\left(P_k, h_k\right)} f(x)\right]_j-\langle \nabla f(x),p^{(j)} \rangle} & \leq \frac{\lambda h_k}{2}\norm{p^{(j)}}^2 \quad \text{a.s.}
\end{split}
\end{equation}
Note that \ref{A_Pas} implies $\|p^{(j)}\|^2\overset{a.s.}{=}d/\ell$ and so  \eqref{eq} and \eqref{lb}  yield
\begin{equation*}
\begin{split}
\norm{\nabla_{\left(P_k, h_k\right)} f(x)-P_k^\top \nabla f(x)}{} & \leq \sqrt{\sum_{j=1}^{\ell}  \left(\frac{\lambda h_k d}{2\ell}\right)^2} = \frac{\lambda h_k d}{2\sqrt{\ell}} \quad\text{ a.s.}.	
\end{split}
\end{equation*}
\qed

\paragraph{Lipschitz smooth functions}
We start with two well-known lemmas on Lipschitz smooth functions, namely differentiable functions with Lipschitz continuous gradient.
\begin{lemma}[Baillon-Haddad Theorem \cite{baillon1977quelques}]\label{BaillHadd}
    Let $f: \R^d \to \R$ be a convex and Fréchet differentiable function with $\lambda$-Lipschitz continuous gradient for some $\lambda>0$. Then $\nabla f$ is $1/\lambda$ co-coercive; namely, for every $x,y\in\R^d$,
    $$\langle \nabla f(y)-\nabla f(x), y-x\rangle \geq \norm{\nabla f(y)-\nabla f(x)}{}^2/\lambda.$$
\end{lemma}
\begin{lemma}[Descent Lemma {\cite[Sec 1.1.2]{Pol87}}]\label{DescentLemm}
	Let $f: \R^d \to \R$ be a Fréchet differentiable function with $\lambda$-Lipschitz continuous gradient. Then, for every $x,y\in \R^d$,
	\begin{equation*}
	\abs{f(y) - f(x) - \langle \nabla f(x), y-x \rangle }\leq \lambda\norm{y-x}{}^2/2.
	\end{equation*}
\end{lemma}
 
\paragraph{Real sequences} 
In this section, we first recall Opial's Lemma in the deterministic setting. In the next, we collect some results regarding convergence and convergence rates for real sequences.

\begin{lemma}[Opial, deterministic version \cite{MR211301}]\label{Opial}
	Let $\mathcal{Z}\subseteq\R^d$ be a non-empty subset and $\left(x_k\right)\subseteq\R^d$ a sequence. Assume that  
	\begin{itemize}
		\item for every $z\in \mathcal{Z}$, 
		$$\exists\lim_{k} \|x_k-z\|;$$
		\item every cluster point of $x_k$ belongs to $\mathcal{Z}$; namely,
		$$x_{k_j}\to x_{\infty} \ \ \implies \ \ x_{\infty} \in \mathcal{Z}.$$
	\end{itemize}
	Then there exists $\bar{x}\in \mathcal{Z}$ such that $x_k\to\bar{x}$.
\end{lemma}
\begin{lemma}[Chung's Lemma \cite{MR64365}]\label{chung}
	Let $\left(a_k\right)$ be a non-negative sequence and $c, \ d$ and $p$ strictly-positive constants.\\
	First suppose that, for every $k\in\N$,
	$$a_{k+1}\leq\left(1-\frac{c}{k}\right)a_k+\frac{d}{k^{p+1}}.$$
	Then
	\begin{equation}
	a_k\leq 
	\begin{cases}
	\frac{d}{\left(c-p\right)k^p}+o\left(\frac{1}{k^p}\right) \ \ \ \ \ & \text{if} \ \ p<c; \\
	 
	O\left(\frac{\log k}{k^{c}}\right) \ \ \ \ \ & \text{if} \ \ p=c;\\
	 
	O\left(\frac{1}{k^c}\right) \ \ \ \ \ & \text{if} \ \ p>c.  	
	\end{cases}
	\end{equation}
	Now suppose that, for some $0<s<1$ and $s<t$ and every $k\in\N$,
	$$a_{k+1}\leq\left(1-\frac{c}{k^s}\right)a_k+\frac{d}{k^{t}}.$$
	Then
	$$a_k\leq \frac{d}{c} \frac{1}{k^{t-s}}+o\left(\frac{1}{k^{t-s}}\right).$$
\end{lemma}
\begin{lemma}\label{lem:aux}
	Let $\left(a_k\right)$ and $\left(c_k\right)$ be non-negative sequences with $\lim_k c_k =0$ and let $\eta\in\left(0,1\right)$. If for every $k\in\N$
	\begin{equation}\label{ineqqq}
	a_{k+1}\leq\left(1-\eta\right) a_k +c_k,
	\end{equation}
	then $\lim_k a_k =0$.\\
	Moreover, if $c_k=c/k^t$ for some $c\geq 0$ and $t>0$, then
	$$\limsup_k \ k^t a_k\leq \frac{c}{\eta}.$$
\end{lemma}
\begin{proof}
    Iterating the inequality in \eqref{ineqqq}, we get
    \begin{equation*}
    \begin{split}
	a_{k} & \leq(1-\eta)^k a_0 \\
	&+ \left(c_{k-1}+(1-\eta) c_{k-2} + (1-\eta)^2 c_{k-3} + (1-\eta)^3 c_{k-4}+...+(1-\eta)^{k-1}c_0 \right)\\
	& \leq a_0 + \left(\sup_{j\geq 0} c_j \right) \sum_{j\geq 0}^{} (1-\eta)^j = : \tilde{C} <+\infty,
    \end{split}
    \end{equation*}
    where we used the fact that $\left(c_k\right)$ is convergent (and thus bounded) and that $\eta\in\left(0,1\right)$. In particular, the sequence $\left(a_k\right)$ is bounded and so its $\limsup$ is a real number. Then, again from the hypothesis that $a_{k+1}\leq(1-\eta) a_k +c_k$ and the existence of $\lim_k c_k$, we have
    \begin{equation*}
    \begin{split}
	\limsup_k a_{k}& \leq \limsup_k \left[(1-\eta) a_k + c_k \right] = (1-\eta) \limsup_k a_k + \lim_k c_k.
    \end{split}
    \end{equation*}
    So, as $\lim_k c_k =0$ and $\limsup_k a_k \in \left[0,+\infty\right)$, $\eta\limsup_k a_k\leq 0$.
    Finally, as $\eta\in\left(0,1\right)$,
    $$\limsup_k a_k\leq 0$$
    and so, as $a_k\geq 0$, $\lim_k a_k = 0$.
    Now assume that $c_k=c/k^t$ for some $c\geq 0$ and $t>0$. By $k^t\leq (k+1)^t$, we have
    $$\frac{c}{k^t}\leq\frac{c}{\eta}\left[\frac{1}{(k+1)^t}-(1-\eta)\frac{1}{k^t}\right].$$
    Using the latter inequality \eqref{ineqqq}, we get
    \begin{equation}\label{innn}
        a_{k+1}-\frac{c}{\eta(k+1)^t}\leq(1-\eta)\left[a_{k}-\frac{c}{\eta k^t}\right].
    \end{equation}
    First suppose that there exists $\bar{k}\in\N$ such that $a_{\bar{k}}\leq \frac{c}{\eta \bar{k}^{t}}$. Using \eqref{innn}, it is easy to see by recursion that, for every $k\geq \bar{k}$,
    $$a_{k}\leq \frac{c}{\eta k^{t}}$$
    and so that the claim holds. Now suppose the opposite; namely, that for every $k\in\N$
    $$a_k-\frac{c}{\eta k^t}>0.$$
    Then, iterating \eqref{innn}, we get
    \begin{equation*}
        a_{k}-\frac{c}{\eta k^t}\leq(1-\eta)^k\left[a_{1}-\frac{c}{\eta}\right].
    \end{equation*}
    Finally,
    $$\limsup_k \ k^t a_k\leq \limsup_k \left\{ \frac{c}{\eta}+k^t(1-\eta)^k\left[a_{1}-\frac{c}{\eta}\right] \right\}=\frac{c}{\eta}.$$
\end{proof}
 
We conclude this part with the following three well-known results. The proof of the first can be found in \cite[Theorem 3.3.1]{knopp1990theory}, while the second is just the deterministic version of Lemma \ref{PC}. For the third, related to estimates with errors, see \cite{schmidt2011convergence}.
\begin{lemma}\label{smallorate}
	Let $\left(a_k\right)$ be a non-negative, non-increasing and summable sequence. Then $a_k=o(k^{-1})$. 
\end{lemma}

\begin{lemma}\label{PC_det} Let $(r_k)$, $(\beta_k)$, $(y_k)$ and $(w_k)$ be non-negative real sequences with $(\beta_k)$ and $(w_k)$ in $\ell^1$. Suppose that, for every $k\in\N$,
	$$r_{k+1} - (1+\beta_k)r_k +y_k \leq  w_k.$$
	Then $(r_k)$ is convergent and $(y_k)$ belongs to $\ell^1$. 
\end{lemma}

\begin{lemma}\label{Bihari}[Discrete Bihari's Lemma]
Assume that $\seq{u_k}$ is a non-negative real sequence, that $\seq{S_k}$ is a non-decreasing sequence such that $S_0\geq u_0^2$ and that $\seq{\rho_j}$ is a non-negative sequence. If, for every $k\in\N$,
	\begin{equation*}
	\begin{split}
	u_{k}^2 \leq S_{k} +\sum_{j=0}^{k} \rho_j u_j,
	\end{split}
	\end{equation*}
then, for every $k\in\N$,
	\begin{equation*}
	u_k\leq \frac{1}{2}\sum_{j=0}^{k}\rho_j+\left[S_{k}+\left(\frac{1}{2}\sum_{j=0}^{k}\rho_j\right)^2\right]^{1/2}.    
	\end{equation*}

\end{lemma}
 
\paragraph{Random sequences} In this section, we recall the extension of Opial's Lemma \ref{Opial} and Lemma \ref{PC_det} to the stochastic setting (see Lemma \ref{StochOpial} and \ref{PC} - respectively). For completeness, we show the proof of Lemma \ref{StochOpial}, starting with the auxiliary Lemma \ref{aux}. In the next, $\left(\Omega, \mathcal{A}, \mathbb{P}\right)$ is a probability space and we say that $\tilde{\Omega}$ is full-measure (f.m.) if $\tilde{\Omega}\subseteq\Omega$ and $\mathbb{P}(\tilde{\Omega})=1$.
% \begin{lemma}\label{auxaux}
% 	Every countable intersection of f.m. subsets is f.m.
% % \end{lemma}
% \begin{proof}
% 	Let $(\Omega_k)$ be a sequence of f.m. subsets. By $\sigma$-subadditivity and since $\mathbb{P}(\Omega_k^C)=0$ for every $k\in\N$,
% 	$$\mathbb{P}(\cap_{k\in\N} \ \Omega_k)=1-\mathbb{P}(\left(\cap_{k\in\N} \ \Omega_k\right)^C)=1-\mathbb{P}(\cup_{k\in\N} \ \Omega_k^C)\geq 1-\sum_{k\in\N} \mathbb{P}(\Omega_k^C)=1.$$
% 	Then, $\mathbb{P}(\cap_{k\in\N} \ \Omega_k)=1$.
% \end{proof}
The proof of the Lemma \ref{aux} can be found in \cite[Proposition 2.3]{combettes2015stochastic} as part of a result about Fejér monotonicity. We repeat the reasoning for clarity.
\begin{lemma}\label{aux}
	Let $\mathcal{Z}\subseteq\R^d$ be a non-empty subset and $(x_k)$ a random sequence on $\left(\Omega, \mathcal{A}, \mathbb{P}\right)$ with values in $\R^d$. Assume that, for every $z\in \mathcal{Z}$, there exists $\Omega_z$ \ f.m. such that, for every $\omega\in\Omega_z$, the sequence $ \seq{\|x_k(\omega)-z\|} $ converges. Then there exists $\tilde{\Omega}$ \ f.m. such that, for every $\omega\in\tilde{\Omega}$ and every $z\in\mathcal{Z}$, $\exists\lim_k \|x_k(\omega)-z\|$.
\end{lemma}
\begin{proof}
	By separability of $\mathcal{Z}\subseteq \R^d$, let $W\subseteq \mathcal{Z}$ be a countable subset such that $\overline{W}=\mathcal{Z}$ and define $\tilde{\Omega}:=\bigcap_{w\in W} \ \Omega_w$. As $W$ is countable and $\mathbb{P}(\Omega_w)=1$ for every $w\in W$, $\tilde{\Omega}$ is f.m. Moreover, for every $\omega\in\tilde{\Omega}$ and every $w\in W$, there exists 
	$$\lim_k \|x_k(\omega)-w\|.$$ 
	We want to show that, for every $\omega\in\tilde{\Omega}$ and every $z\in \mathcal{Z}$, there exists 
	$$\lim_k \|x_k(\omega)-z\|.$$ 
	Fix $\omega\in\tilde{\Omega}$ and $z\in \mathcal{Z}$. As $W$ is dense in $\mathcal{Z}$, there exists a sequence $w_j\subseteq W$ such that $w_j\to z.$ As $w_j\in W$ for each $j\in\N$, we know that there exists 
	\begin{equation}\label{lim}
	\lim_k \|x_k(\omega)-w_j\|=:\tau_j(\omega).
	\end{equation}
	Notice that 
	\begin{equation}\label{ab}
	-\|w_j-z\|\leq \|x_k(\omega)-z\|-\|x_k(\omega)-w_j\|\leq \|w_j-z\|.
	\end{equation}
	Then,
	\begin{equation*}
	\begin{split}
	-\|w_j-z\| & \leq \\
	(\ref{ab}) \ & \leq \liminf_k \left[\|x_k(\omega)-z\|-\|x_(\omega)-w_j\|\right]\\
	(\ref{lim}) \ & =\liminf_k \|x_k(\omega)-z\| - \tau_j(\omega)\\
	& \leq \limsup_k \|x_k(\omega)-z\| - \tau_j(\omega)\\
	(\ref{lim}) \ &= \limsup_k \left[\|x_k(\omega)-z\|-\|x_k(\omega)-w_j\|\right]\\
	(\ref{ab}) \  &\leq \|w_j-z\| .
	\end{split}
	\end{equation*}
	Taking the limit for $j\to +\infty$ and recalling that $w_j\to z$, 
	$$\liminf_k \|x_k(\omega)-z\| = \limsup_k \|x_k(\omega)-z\| $$
	and so that there exists $\lim_k \|x_k(\omega)-z\|$.
\end{proof}
\begin{lemma}[Opial, stochastic version]\label{StochOpial}
	Let $\mathcal{Z}\subseteq\R^d$ a non-empty subset and $(x_k)$ a random sequence on $\left(\Omega, \mathcal{A}, \mathbb{P}\right)$ with values in $\R^d$. Assume that
	\begin{itemize}
		\item for every $z\in \mathcal{Z}$, there exists $\Omega_z$ \ f.m. such that, for every $\omega\in\Omega_z$,
		$$\exists\lim_k \|x_k(\omega)-z\|;$$
		(i.e., for every $z\in \mathcal{Z}$, the random variable $\|x_k-z\|$ converges a.s.)
		\item there exists $\hat{\Omega}$ \ f.m. such that, for every $\omega\in \hat{\Omega}$, every cluster point of $x_k(\omega)$ belongs to $\mathcal{Z}$; namely,
		$$x_{k_j}(\omega)\to x_{\infty} \ \ \implies \ \ x_{\infty} \in \mathcal{Z}.$$
	\end{itemize}
	Then there exists a $\mathcal{Z}$-valued random variable $\bar{x}$ such that $x_k\to\bar{x}$ a.s.; namely, there exists $\bar{\Omega}$ \ f.m. such that, for every $\omega\in \bar{\Omega}$, $x_k(\omega)\to\bar{x}(\omega)$ with $\bar{x}(\omega)\in\mathcal{Z}$.
\end{lemma}
\begin{proof}
	From the assumptions and Lemma \ref{aux}, there exists $\tilde{\Omega}$ \ f.m. such that, for every $\omega\in\tilde{\Omega}$ and every $z\in\mathcal{Z}$,
	$$\exists\lim_k \|x_k(\omega)-z\|.$$
	Let $\bar{\Omega}:=\tilde{\Omega}\cap \hat{\Omega}$. Then $\mathbb{P}(\bar{\Omega})=1$ and, for every $\omega \in \bar{\Omega}$, we have both that 
	\begin{itemize}
		\item for every $z\in \mathcal{Z}$, $\exists\lim_k \|x_k(\omega)-z\|$;
		\item every cluster point of $x_k(\omega)$ belongs to $\mathcal{Z}$.
	\end{itemize}
	We conclude by the deterministic version of Opial's Lemma \ref{Opial} that, for every $\omega\in \bar{\Omega}$, there exists $\bar{x}\left(\omega\right)\in\mathcal{Z}$ such that $x_k(\omega)\to\bar{x}(\omega)$.
\end{proof}

\begin{lemma}[Robbins-Siegmund \cite{MR0343355}]\label{PC} Let $(\Omega, \mathcal{A}, \mathbb{P})$ be a measure space, and let  $(\mathcal{F}_k)$ 
be a filtration of $\mathcal{A}$. Let $(r_k)$, $(y_k)$, $(w_k)$ and $(\beta_k)$ be sequences of non-negative random variables adapted to  $(\mathcal{F}_k)$. Let $(\beta_k)$ and $(w_k)$ belong to $\ell^1$ a.s. and suppose that, for every $k\in\N$,
	$$\mathbb{E} \left[r_{k+1} \mid  \mathcal{F}_k\right] - (1+\beta_k)r_k +y_k \leqas  w_k.$$
	Then $(r_k)$ converges a.s. to a random variable with non-negative values and $\seq{y_k} \in \ell^1$ a.s. 
\end{lemma}
\begin{remark}\label{lemma: integrabilityI}
Consider a function $f: \reals^d \to \reals$ with $\lambda$-Lipschitz gradient and at least one minimizer. Applying the recursion of Algorithm \ref{eqn: derivative-free} to such a function with an arbitrary starting point $x_0 \in \reals^d$ we get that $\norm{x_k-x_*}^2$, $\langle P_kP_k^\top \nabla f_k, x_k-x_* \rangle$,  and $\norm{\nabla f_k}^2$ are bounded for all $k>0$. In particular, $\norm{\nabla f_k}^2$, $\norm{x_k-x_*}^2$, and $\langle P_kP_k^\top \nabla f_k, x_k-x_* \rangle$ are all integrable.  To see this note first that $\nabla f_k$ and $x_k$ are both measurable. Recall that for finite $x_k$, $\lambda$-Lipschitz gradient of $f$ implies that $\norm{\nabla f_k}^2 \leq \lambda^2 \norm{x_k-x_*}^2$. Choose an arbitrary finite $x_0$ to begin the recursion. Then,
    \begin{align*}
        \norm{x_1 -x_*}^2 &\leq 2\norm{x_0-x_*}^2 + 2\alpha_0^2\norm{P_0\nabla_{P_0} f_0}^2 \\ &\leq C_1 + C_2\norm{P_0\nabla_{P_0} f_0 - P_0P_0^\top \nabla f_0 + P_0P_0^\top \nabla f_0}^2  \\
        \mathrm{Lemma }\,\, \ref{lem}&\leq C_1 + C_2\left(C_3 + \left (\frac{d}{\ell}\right)^2 \norm{\nabla f_0}^2\right),
    \end{align*}
where $C_1, C_2, C_3$ are fixed and finite, the values are unimportant but can be calculated. Repeated recursion reveals that $\norm{x_k-x_*}^2$ is bounded for all $k>0$. Further, by $\lambda$-Lipschitz gradient, $\norm{\nabla f_k}^2 \leq \lambda^2 \norm{x_k-x_*}^2$. The claim on the inner product follows from Young's inequality and the previous results. Finally, since  $\norm{\nabla f_k}^2$, $\norm{x_k-x_*}^2$, and $\langle P_kP_k^\top \nabla f_k, x_k-x_* \rangle$ are bounded, and measurable, they are integrable. 
\end{remark}
\subsection{Special cases of the algorithm}\label{sect: special cases}
We present several well-known special cases of Algorithm \ref{eqn: derivative-free} (among which are: discrete gradient descent, discrete coordinate descent, spherical smoothing, and more), and provide some historical perspective on the development of these black-box algorithms. 

Assumptions \ref{A_Pas} and \ref{A_PE} describe matrices $P$ that are generalizations of a few well-known cases. It is important that $P$ is comprised of $\ell$ orthonormal columns, and in the special case that $d=\ell$, $P$ is an orthogonal matrix. 
Specifically, when $d=\ell$, $PP^\top \nabla f(x) = \nabla f(x)$ and $\nabla_{(P,h)}f(x)$ is a forward finite difference estimate of the gradient along $d$ orthogonal directions. 
Therefore gradient descent and discrete gradient descent can be viewed as special cases of \eqref{eqn: gradient-free} and Algorithm \ref{eqn: derivative-free} respectively. 
Though this is of little significance in practice, it provides a means for verifying our analysis: by setting $\ell=d$ we ought to recover previously stated results for the discrete gradient method.

In our analysis we do not differentiate between various choices of $P$, all of our proofs hold whenever \ref{A_Pas} and \ref{A_PE} are satisfied.
It is clear that the specific choice of $P$ does impact the performance of the algorithm (see, e.g., \cite{berahas2019theoretical,Kozak2021Stochastic}), but the purpose of this work is to present a unified convergence analysis rather than to investigate the nuances of each particular case. 
Therefore, in this section we present several choices for $P$ that satisfy \ref{A_Pas} and \ref{A_PE} with an emphasis on choices that correspond to previously described methods.
Our results hold for all of the special cases described in this section, and in many of the cases our results represent an advancement over the current theoretical understanding of the special cases described.  

\paragraph{Coordinate descent}
Suppose $P=(\sqrt{d/\ell})S$ where $S$ consists of $\ell$ columns of the identity matrix chosen uniformly at random.
In this case it is straightforward to see that $PP^\top\nabla f(x)$ corresponds to $\ell$ coordinate directions of the gradient scaled by a constant $d/\ell$; that is, \eqref{eqn: gradient-free} with this choice of P is a scaling of block coordinate descent with uniform sampling of the coordinates. 
Analogously, $P\nabla_{(P,h)}f(x)$ is (up to a scaling constant) a forward finite difference estimate of the gradient along $\ell$ coordinate directions and we recover discretized coordinate descent. 

\paragraph{Coordinate descent with a change of basis}
More generally, if $P$ consists of $\ell$ columns  selected uniformly at random from an orthogonal matrix in $\reals^{d\times d}$, and scaled by $\sqrt{d/\ell}$ then $P$ satisfies \ref{A_Pas} and \ref{A_PE}.
For instance, suppose that one selects $\ell$ columns uniformly at random from scaled versions of a discrete cosine transform matrix or the Hadamard matrix. The resulting $d \times \ell$ matrix defines a valid matrix $P$ for our analysis.
Any particular choice of fixed orthogonal matrix for $P$ amounts to discretized coordinate descent in a different basis, but if little is known a priori about the structure of $f$ then there is no reason to select one basis over another.

\paragraph{Random orthogonal matrices}
The orthogonal matrices mentioned in the previous paragraph have been coupled with a random component and used extensively to "sketch" problems. Typically, sketching entails approximating a problem by representing the data in a lower dimensional (random) subspace, and solving the approximate problem (see, e.g., \cite{MR2639236,MR2754850} for the least squares case, or \cite{MR3285427,MR4189294} for a more general overview of sketching algorithms in numerical linear algebra). 
The properties of the matrix that projects the data onto a subspace allow for guarantees on the quality of the approximated solution as compared to the true solution. These same matrices can instead be used to sketch the gradient using our method. For an example of such a sketching matrix consider a Hadamard matrix $H \in \reals^{d\times d}$ , a diagonal matrix $D \in \reals^{d \times d}$ with equiprobable entries $\{1,-1\}$ along the diagonal, and a matrix $S \in \reals^{d \times \ell}$ independent of $D$ with columns chosen uniformly at random from the identity. 
Let $P = (\sqrt{d/\ell})DHS$, with $D$ and $S$ re-sampled at each iteration. This is similar to the sketching matrix described in \cite{MR2754850} and its properties are well known, it is simple to verify that it satisfies \eqref{A_Pas} and \eqref{A_PE}. This method is described in \cite{pmlr-v80-choromanski18a}, however they provide only empirical results, making no claims about the convergence properties. To our knowledge our analysis is the first to provide convergence analysis for these types of matrices used in a derivative-free optimization setting. 

\paragraph{Spherical smoothing}
Consider instead $P = (\sqrt{d/\ell})\,Q\I_{d\times\ell}$, where $Q$ is as in the $QR$-decomposition of a matrix $Z = QR \in \reals^{d \times d}$  with $R_{ii} >0,$ and the entries of $Z_k$ are iid $\mathcal{N}(0,1)$.
The matrix $\I_{d \times \ell}$ truncates $Q$ to its first $\ell$ columns so $Q\I_{d\times\ell}$ corresponds to $\ell$ columns of the random orthogonal matrix distributed according to the Haar measure on orthogonal matrices \cite{Mezzadri}. 
In other words, the columns $p^{(j)}$ are orthogonal and distributed uniformly on the sphere for all $j$. Thus, when $\ell=1$, $P\nabla_{(P,h)}f(x)$ is a spherical smoothing estimate of the gradient, as described in, e.g., \cite{MR2298287,berahas2019theoretical}. 
The matrix $Z$ is re-sampled at each iteration so, as with the matrices described in the previous paragraph, the basis changes with each iteration. In fact, sampling from the Haar measure on the set of orthogonal matrices corresponds to sampling uniformly from the set of orthogonal matrices.

For the case $\ell >1$ it is more common in the literature $\cite{berahas2019theoretical}$ to sample $p^{(j)}$ independently and uniformly on the sphere, but in our case, to satisfy Assumptions \eqref{A_Pas} and \eqref{A_PE}, the columns of $P$ must be orthonormal, consistent with \cite{Kozak2021Stochastic}. 
The advantage of a matrix $P$ with orthonormal columns is discussed at length in \cite{Kozak2021Stochastic}, we remark here merely that this property is required to obtain our results and to connect Algorithm \ref{eqn: derivative-free} with discrete gradient descent when $\ell=d$; indeed consider that when $p^{(j)}$ are sampled independently and uniformly on the sphere, which we denote as $p^{(j)} \overset{iid}{\sim} \mathcal{U}(S(0,1))$, the gradient estimate is

$$\nabla f(x) \approx  \frac{1}{\ell}\sum_{j=1}^\ell p^{(j)} \frac{f(x+p^{(j)} h)-f(x)}{h}, \qquad p^{(j)} \overset{iid}{\sim} \mathcal{U}(S(0,1)), 
$$ 
with the discrete gradient recovered only as $\ell \to \infty$. In contrast, the approximation

$$  \nabla f(x) \approx P \begin{pmatrix}
    \frac{f(x_k + p^{(1)}h) -f (x)}{h} \\ \vdots \\
    \frac{f(x_k + p^{(\ell)}h) -f (x)}{h}
    \end{pmatrix}, \qquad p^{(j)} \mathrm{\,columns\, of\, Haar}$$
and the discrete gradient is recovered whenever $\ell=d$ due to the orthogonality of the Haar distributed random matrix.

\paragraph{Gaussian smoothing}
The Gaussian smoothing framework first described in the technical report \cite{nesterov2011random} and later in \cite{nesterov2017random} \emph{does not} fit into our framework because the columns do not have unit norm.
% However, we can re-create a mathematically identical method that does fit into our framework for the gradient-free case.
% Consider a matrix $G\in \reals^{d \times \ell}$ such that each column of $G$ is an orthogonalized (but not orthonormalized) column of $Z$ where each element of $Z$ is  drawn independently from $\mathcal{N}(0,1)$, and re-sampled at each iteration. Let $P^{GS} = (\sqrt{d/\ell})G$. 
% Then, when $\ell=1$, \eqref{eqn: derivative-free} and \eqref{eqn: gradient-free} are exactly Gaussian smoothing as described in \cite{nesterov2017random} when $h>0$ and $h=0$ respectively. But as mentioned in this form it does not satisfy \eqref{A_Pas}.
% For the gradient-free case we can define $P = P^{GS}/\norm{G}$ and multiply the stepsize at each iteration by $\norm{G}^2$, in which case $P$ satisfies \eqref{A_Pas} and \eqref{A_PE}. 
% In the derivative-free case this trick does not work, and does not appear to be possible to put Gaussian smoothing in our framework.
However, spherical smoothing, which is covered by our framework, can be thought of as a normalized version of Gaussian smoothing since a Haar distributed random matrix is generated by orthonormalizing a Gaussian random matrix. It is shown in \cite{berahas2019theoretical} that spherical smoothing provides better approximations to the gradient than does Gaussian smoothing. 

The intuition behind this statement is both illuminating and simple to provide. In Gaussian smoothing, $P = \sqrt{d}z$ where $z \sim \mathcal{N}(0,1)$. Then, $\nabla f(x) \approx (1/h) (f(x + zh)-f(x) )z$. Since $z$ has infinite support, the finite difference stepsize varies with each iteration irrespective of the value of $h$. Thus the approximation of the gradient has a positive probability of being arbitrarily bad even when the direction chosen is near to the direction of the gradient! With spherical smoothing, the directions chosen are identical, but the finite difference stepsize is always of length $h$ which leads to more consistent and predictable results. 

\begin{remark}
The literature is scarce but growing when it comes to convergence results for many of the above-mentioned methods in the finite difference setting. There has been plenty of attention to analyzing these methods when exact directional derivatives are available (i.e., the setting of \eqref{eqn: gradient-free}, see, e.g., \cite{MR3347548,Kozak2021Stochastic,MR3116649,MR1264365}), but implementing the algorithms they analyze requires access to exact directional derivatives (e.g., via forward-mode automatic differentiation). Generally speaking, practical implementations of these algorithms often do not use automatic differentiation software -- either because it is not feasible, or because it is too restrictive and time consuming -- relying instead on function evaluations and finite difference approximations of the gradient.
\end{remark}

\subsection{Previous works}
The limit definition of the derivative makes it natural to estimate the gradient via finite differences, the method of finite difference gradient descent goes back to Cauchy \cite{cauchy1847methode}. 
For a more modern treatment we can look to the seminal paper of Kiefer and Wolfowitz \cite{MR50243} which extends the results of Robbins and Monro \cite{MR42668} on stochastic approximation to the case where the gradient is approximated by a central finite difference. 
In \cite{MR50243} it is shown that with sufficiently fast decaying stepsize and finite difference step, the iterates converge asymptotically to the minimizer of a function under regularity conditions on the specified function.
The setting of \cite{MR50243} differs from that of this paper by accessing only stochastic approximations of the function that is being minimized whereas we assume the function can be queried exactly, allowing for much stronger results.

Kushner and Clark \cite[pg. 59-61]{MR499560} explore the asymptotic properties of what is now known as spherical smoothing, a special case of our algorithm.
They work in a slightly different setting, adopting the noisy function evaluations case of Kiefer and Wolfowitz and making more assumptions on the objective function such as twice-differentiability. Again, the results are weaker and less general than those we provide. For their algorithm and under their regularity assumptions, Kushner and Clark are able to show that the iterates of their algorithm converge to a minimizer. Their analysis provided the basis for much subsequent work. 
Spall \cite{MR1148715} compares his work to that of Kushner and Clark with the notable difference being that Spall does not choose directions uniformly on the sphere, but from a more general, unspecified, mean-zero distribution; this is perhaps the clearest intellectual predecessor to Gaussian smoothing, which is discussed below, because it includes Gaussian smoothing as a special case. The generality of Spall's results requires him to assume thrice-differentiability of the objective function.
A main beneficiary of the work of Kushner and Clark is the reinforcement learning community which has adopted their method and renamed it \emph{evolutionary strategies}. 
These evolutionary strategies were first described by Williams \cite{williams1992simple}, who called them REINFORCE algorithms. Williams was apparently unaware of the work of Kushner and Clark and suggested that "While there is a clear need for an analytical characterization of the asymptotic behavior of REINFORCE algorithms, such results are not yet available, leaving simulation studies as our primary source of understanding of the behaviour of these algorithms". Subsequent literature that makes use of evolutionary strategies (e.g., \cite{pmlr-v80-choromanski18a,salimans2017evolution}) typically mention Williams as the forebear for these methods, while some mention Spall's work in providing convergence properties, and many mention the work of Nesterov \cite{nesterov2011random} discussed below. 
Of particular note is \cite{MR2298287}, which works in the setting of Kushner and Clark but frames it as a reinforcement learning problem and provides finite-time results in expectation. See also \cite{Agarwal}. 

 The asymptotic behavior of discrete gradient descent with exact function queries is investigated in \cite{MR1264365} which provides an upper bound on the level sets of the limiting function values, as well as a radius of convergence for the iterates. 
 The setting is somewhat restrictive, with the objective $f$ assumed to be $\gamma$-strongly convex with $\lambda$ Lipschitz gradient, and only asymptotic properties are provided.
A method for performing finite difference coordinate descent is described and analyzed in  \cite{MR1264365}, however the method described therein uses coordinates only as directions, no approximate derivative (i.e., finite difference) information is used.
 
Nesterov published a technical note \cite{nesterov2011random} in 2011, and later a peer-reviewed article with Spokoiny \cite{nesterov2017random} analyzing the convergence behaviour of a finite difference optimization algorithm in which the direction of descent is chosen according to a Gaussian distribution in the following manner.
Let $U \sim \mathcal{N}(0, \I_d)$, then the direction of descent is $U (f(x+Uh)-f(x)) / h \approx \nabla f(x)$.

Though they do not analyze any particular algorithm, Berahas et al. \cite{berahas2019theoretical} provide a thorough comparison of the quality of various derivative-free approximations of the gradient.
Included in their analysis are finite difference (and by simple corollary, coordinate descent), Gaussian smoothing, and spherical smoothing. 
Their work provides the framework required for analysis of any gradient-based descent algorithm based on finite difference approximations,
 analysis that is particularly useful in the related paper \cite{berahas2019global} which describes a line search method appropriate for gradient descent algorithms when the gradient is known only approximately.
Expected rates of convergence are provided in \cite{berahas2019global} under a variety of convexity assumptions when the gradient is approximated using any of the methods discussed in \cite{berahas2019theoretical}. 

The algorithm analyzed in \cite{Kozak2021Stochastic} is identical to \eqref{eqn: gradient-free}. They discuss but do not analyze Algorithm \eqref{eqn: derivative-free}, which is the primary focus of our work. 
Furthermore, we provide stronger results in the convex case using \eqref{eqn: gradient-free}: we prove almost sure convergence of the iterates to a minimizer (cf. Theorem \ref{th:FejerSStrong}) whereas in \cite{Kozak2021Stochastic} only expected results are provided, and only for the function values.

Trust-region methods are a class of derivative-free optimization algorithms that we do not explore here, we note merely that recent work analyzes a stochastic subspace method that is analogous to ours for the trust-region framework, specifically for non-linear least squares problems \cite{cartis2021scalable} .

\section*{Acknowledgements}
L.R. acknowledges support from the Center for Brains, Minds and Machines (CBMM), funded by NSF STC award CCF-1231216, and the Italian Institute of Technology. L.R. also acknowledges the financial support of the European Research Council (grant SLING 819789), the AFOSR projects FA9550-
18-1-7009, FA9550-17-1-0390 and BAA-AFRL-AFOSR-2016-0007 (European Office of Aerospace Research and Development), and the EU H2020-MSCA-RISE project NoMADS - DLV-777826. This material is based upon work supported by the Air Force Office of Scientific Research under award number FA8655-20-1-7028. This work has been supported by the ITN-ETN project TraDE-OPT funded by the European Union’s Horizon 2020 research and innovation programme under the Marie Skłodowska-Curie grant agreement No 861137.

\end{document}